\documentclass{amsart}

\usepackage{max}
\usepackage[square,sort,comma,numbers]{natbib}

\newcommand*{\MCG}{\mathbf{MCG}}
\newcommand*{\Braid}{\mathbf{Braid}}
\newcommand*{\Sym}{\mathbf{Sym}}

\newcommand*{\RO}{R_{\mathrm{or}}}

\title[
    Secondary stability for nonorientable surfaces
]{
    Secondary homological stability for mapping class groups of nonorientable surfaces
}
\author{Max Vistrup}
\date{May 29, 2021}

\begin{document}

\begin{abstract}
    Using the Galatius--Kupers--Randal-Williams framework of cellular $E_2$-algebras,
        we prove a secondary stability theorem for mapping class groups of nonorientable surfaces.
    As a corollary,
        we obtain a new best known stability range for
            the homology of the mapping class groups of nonorientable surfaces
            with respect to adding torus holes.
\end{abstract}

\maketitle

\tableofcontents
\newpage

%! TEX root = main.tex

\section{Introduction}
\label{sec:introduction}

For a compact surface $S$, orientable or not,
    let $\Homeo^\del(S)$ be
        the space of homeomorphisms of $S$
            fixing $\del S$ pointwise
            equipped with the compact--open topology.
Composition of homeomorphisms endows $\Homeo^\del(S)$
    with the structure of a topological group.
The \emph{mapping class group} of $S$ is the group of path components
    \[
        \Gamma(S) \defeq \pi_0 \Homeo^\del(S);
    \]
    that is, the group of isotopy classes of boundary-fixing homeomorphisms of $S$.
For each $g \geq 1, r \geq 0$,
    let
    \[
        N_{g, r}
        \defeq
        (\RP^2)^{\# g} - (\Int(D^2) \times \{1, \ldots, r\})
    \]
    denote the nonorientable surface of \emph{genus} $g$ with $r$ boundary components.
Similarly, let $S_{g, r}$ denote
    the orientable surface of genus $g$ with $r$ boundary components.
As a convention,
    we let
    $
        N_{0, r} \defeq S_{0, r}
    $
    be the $r$-punctured disk.
For convenience,
    we write
    $
        \Gamma_{g, r} \defeq \Gamma(S_{g, r})
    $
    and
    $
        N\Gamma_{g, r} \defeq \Gamma(N_{g, r}).
    $

For $g \geq 1$,
    there are \emph{stabilization maps}
    $
        \Gamma_{g - 1, 1} \to \Gamma_{g, 1}
    $
    given by
        extending a homeomorphism of $S_{g - 1, 1}$
            along $S_{g - 1, 1} \incl S_{g, 1}$
            putting the identity $\id_{S_{1, 2}}$ outside $S_{g - 1, 1}$.
Harer \cite{Harer85} proved that the groups $\Gamma_{g, 1}$ exhibit \emph{homological stability}
    with respect to the genus $g$.
Specifically, he proved that $H_d(\Gamma_{g, 1}, \Gamma_{g - 1, 1})$ vanishes
    in a range of $d$ increasing with $g$
    by slope $\frac13$.
Subsequent papers have improved this range.
Ivanov \cite{Ivanov89} attained a slope $\frac12$ range.
Boldsen \cite{Boldsen12} attained a slope $\frac23$ range,
    or
    $
        d \leq \floor{\frac{2g - 2}{3}}
    $
    to be precise
    (see \cite{Wahl13} for an exposition of the proof).
More recently,
    Galatius, Kupers, and Randal-Williams \cite[\thm{B(i)}]{GKRW19MCG} attained the range
    $
        d \leq \floor{\frac{2g - 1}{3}}
        \iff
        \frac{d}{g} < \frac23.
    $

A similar story transpired for
    mapping class groups $N\Gamma_{g, 1}$ of nonorientable surfaces.
In this case, there are two relevant stabilization maps
    that increase genus.
First, for $g \geq 3$,
    there is a \emph{torus hole stabilization map}
    \begin{equation}\label{eqn:torus_hole_stabilization}
        N\Gamma_{g - 2, 1} \to N\Gamma_{g, 1},
    \end{equation}
    defined as above by taking the boundary sum with the punctured torus $S_{1, 1}$.
Second, for $g \geq 1$,
    there is a \emph{crosscap stabilization map}
    \begin{equation}\label{eqn:crosscap_stabilization}
        N\Gamma_{g - 1, 1} \to N\Gamma_{g, 1},
    \end{equation}
    defined as before but by taking the boundary sum
        with the Möbius strip $N_{1, 1}$ (a copy of which is known as a \emph{crosscap})
        instead of $S_{1, 1}$.
These maps are well-defined at least up to an inner automorphism of the target,
    which means that there is no ambiguity on group homology.

Wahl \cite{Wahl07} was the first to prove homological stability for $N\Gamma_{g, 1}$.
She proved homological stability with respect to the crosscap stabilization map,
    showing that
    $
        H_d(N\Gamma_{g, 1}, N\Gamma_{g - 1, 1}) = 0
    $
    in a range of $d$ increasing with $g$ with slope $\frac14$.
In terms of Euler characteristic,
    the genus of an orientable surface is worth twice the genus of a nonorientable surface,
    and thus this range is analogous to Ivanov's range
        for the orientable case discussed above.
Randal-Williams \cite[1.4]{RandalWilliams16} improved Wahl's range to a slope $\frac13$ range,
    which in turn is analogous to the improvement of Boldsen
        in the orientable case mentioned above.

Moreover, Randal-Williams \emph{loc.\,cit.}\ proved homological stability
    with respect to the torus hole stabilization map \cref{eqn:torus_hole_stabilization},
    showing that
    $
        H_d(N\Gamma_{g - 2, 1})
        \to
        H_d(N\Gamma_{g, 1})
    $
    is an isomorphism for $g \geq 3d + 6$.
We prove that slightly outside this range,
    there may not be stability,
    but a secondary stability phenomenon
        relating torus hole stabilization to crosscap stabilization
        occurs.
The crosscap stabilization maps \cref{eqn:crosscap_stabilization} commute
    with the torus hole stabilization maps \cref{eqn:torus_hole_stabilization}
    up to an inner automorphism of $N\Gamma_{g, 1}$.
For each choice of such inner automorphism,
    there is an induced map on the relative homology.
There is a preferred choice of such induced map,
    which we describe in \cref{subsec:secondary_stabilization_map}.

\begin{maintheorem}\label{thm:secondary_stability}
    Let $g \geq 4$.
    If $g$ is odd,
        the secondary stabilization map
        (cf.~\cref{def:map1})
        \[
            H_d(N\Gamma_{g - 1, 1}, N\Gamma_{g - 3, 1})
            \oplus
            H_d(\Gamma_{(g - 1)/2, 1}, \Gamma_{(g - 3)/2, 1})
            \to
            H_d(N\Gamma_{g, 1}, N\Gamma_{g - 2, 1})
        \]
        is
            surjective if $\frac{d}{g} < \frac13$ and
            an isomorphism if $\frac{d + 1}{g} < \frac13$.
    If $g$ is even,
        the same is true for the secondary stabilization map
        \[
            H_d(N\Gamma_{g - 1, 1}, N\Gamma_{g - 3, 1})
            \to
            H_d(N\Gamma_{g, 1}, N\Gamma_{g - 2, 1}).
        \]
\end{maintheorem}

We prove this result using the framework of cellular $E_k$-algebras
    developed by Galatius, Kupers, and Randal-Williams \cite{GKRW19Ek}.
Results like \cref{thm:secondary_stability} about the \q{stability of the failure of stability}
    are called \emph{secondary stability theorems}.
Another result of this kind was obtained
    by Galatius, Kupers, and Randal-Williams in \cite{GKRW19MCG}
    which inspired this paper.
There is a qualitative difference
    between their secondary stability result and ours
    in the fact that our secondary stabilization map does not increase the homological degree.

Combining the injectivity part of \cref{thm:secondary_stability} with
    Harer stability \cite[\thm{B(i)}]{GKRW19MCG} and
    homological stability with respect to the torus hole stabilization maps
        \cite[1.4(i) $+$ 1.4(ii)]{RandalWilliams16},
    by repeatedly applying the secondary stabilization map
        until entering the stable range,
    we obtain the corollary,

\begin{maincorollary}\label{cor:torus_hole_stability}
    Let $g \geq 3$.
    If $\frac{d + 1}{g} \leq \frac13$,
        then
        $
            H_d(N\Gamma_{g, 1}, N\Gamma_{g - 2, 1}) = 0.
        $
\end{maincorollary}

This corollary improves the range of Randal-Williams \emph{loc.\,cit.}
However, it does not extend the range in which the groups $H_d(N\Gamma_{g, 1})$ are stable
    with respect to \cref{eqn:torus_hole_stabilization},
    as Randal-Williams already proves that the maps within his range are isomorphisms.
Rather, the corollary extends his result by showing that right outside his range
    the stabilization maps are surjective,
    as is a typical pattern in algebraic topology.

\begin{figure}
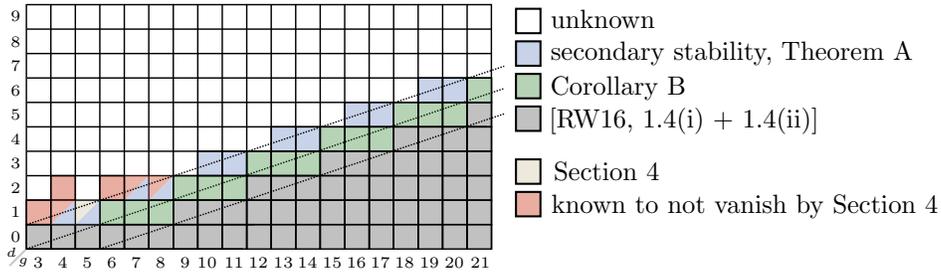

    {
        \tiny
        \incfig[\textwidth]{torus_hole_stability}
    }
    \caption{
        A table showing the known vanishing and secondary stability ranges of
            $
                H_d(N\Gamma_{g, 1}, N\Gamma_{g - 2, 1}),
            $
            for small $g$.
    }
    \label{fig:torus_hole_stability}
\end{figure}

\cref{thm:secondary_stability} admits a sister statement,
    in which the roles of crosscaps and torus holes have been switched:

\begin{maintheorem}\label{thm:secondary_stability_2}
    Let $g \geq 4$ be an even integer.
    Then the secondary stabilization map
        (cf.~\cref{def:map2})
        \[
            H_d(N\Gamma_{g - 2, 1}, N\Gamma_{g - 3, 1})
            \to
            H_d(N\Gamma_{g, 1}, N\Gamma_{g - 1, 1})
        \]
        is
            surjective if $\frac{d}{g} < \frac13$ and
            an isomorphism if $\frac{d + 1}{g} < \frac13$.
\end{maintheorem}

For $g$ odd,
    a more complicated statement may be extracted from \cref{thm:quotient_vanishing}
    (cf.~\cref{rem:g_odd}).
As with \cref{thm:secondary_stability},
    \cref{thm:secondary_stability_2} is superseded by actual stability in most cases:

\begin{figure}
    {
        \tiny
        \incfig[\textwidth]{crosscap_stability}
    }
    \caption{
        A table showing the known vanishing and secondary stability ranges of
            $
                H_d(N\Gamma_{g, 1}, N\Gamma_{g - 1, 1}),
            $
            for small $g$.
    }
    \label{fig:stability_theorems}
\end{figure}

\subsection{The secondary stabilization maps}
\label{subsec:secondary_stabilization_map}

We now describe the secondary stabilization maps.
In the following,
    we shall work with surfaces inside $[0, n] \times I \times \R^\infty$
    that have their boundary in $\del ([0, n] \times I) \times 0$
    and contain a neighborhood
    $
        \del_\epsilon ([0, n] \times I) \times 0 \subset I^2 \times 0
    $
    of this boundary.
Boundary sum $\oplus$ of such surfaces is given by horizontal juxtaposition.
Fix such models
    $
        N_{1, 1}
        \subseteq
        I^2 \times \R^\infty
    $
    and
    $
        S_{1, 1}
        \subseteq
        I^2 \times \R^\infty
    $
    for the Möbius strip and the punctured torus respectively,
    and define
        $N_{g, 1} \defeq N_{1, 1}^{\oplus g} \subseteq [0, g] \times I \times \R^\infty$ and
        $S_{h, 1} \defeq S_{1, 1}^{\oplus h} \subseteq [0, h] \times I \times \R^\infty$.
% TODO: define the clockwise half Dehn twist
For a surface $S \subseteq I^2 \times \R^\infty$
    again subject to the aforementioned conditions,
    the square
    \begin{equation}\label{dia:stabilization_up_to_braiding}
        \begin{tikzcd}[column sep = large]
            \Gamma(S)
            \ar[rr, "\oplus \id_{N_{1, 1}}"]
            \ar[d, "\oplus \id_{S_{1, 1}}"]
                &
                    & \Gamma(S \oplus N_{1, 1})
                    \ar[d, "\oplus \id_{S_{1, 1}}"]
            \\
            \Gamma(S \oplus S_{1, 1})
            \ar[r, "\oplus \id_{N_{1, 1}}"]
                & \Gamma(S \oplus S_{1, 1} \oplus N_{1, 1})
                \ar[r, "\Gamma(\id_S \oplus \beta)"]
                    & \Gamma(S \oplus N_{1, 1} \oplus S_{1, 1})
        \end{tikzcd}
    \end{equation}
    commutes,
    where
    $
        \beta :
            N_{1, 1} \oplus S_{1, 1}
            \to
            S_{1, 1} \oplus N_{1, 1}
    $
    is the clockwise half Dehn twist (see \cref{fig:elementary_braid}),
    which is well-defined up to isotopy,
    and $\Gamma(\id_S \oplus \beta)$ means conjugation by $\id_S \oplus \beta$.

For each homeomorphism
    $
        S_{h, 1} \oplus N_{1, 1}
        \cong
        N_{2h + 1, 1},
    $
    we get an isomorphism of pairs of groups
    \begin{equation}\label{eqn:orientable_nonorientable_sum_identification}
        (\Gamma(S_{h, 1} \oplus N_{1, 1} \oplus S_{1, 1}), \Gamma(S_{h, 1} \oplus N_{1, 1}))
        \cong
        (\Gamma(N_{2h + 1, 1} \oplus S_{1, 1}), \Gamma(N_{2h + 1, 1})).
    \end{equation}
Counting the choice of the homeomorphism,
    this isomorphism is ambiguous only
        up to conjugation by an element in $\Gamma(N_{2h + 1, 1})$.
% The choice of the diffeomorphism affords
%     an ambiguity of this isomorphism up to
%         conjugation by an element in $\Gamma(N_{2h + 1, 1})$.
Consequently, there is a well-defined isomorphism
    \begin{equation}\label{eqn:orientable_secondary_iso}
        H_d(\Gamma(S_{h, 1} \oplus N_{1, 1} \oplus S_{1, 1}), \Gamma(S_{h, 1} \oplus N_{1, 1}))
        \cong
        H_d(\Gamma(N_{2h + 1, 1} \oplus S_{1, 1}), \Gamma(N_{2h + 1, 1})).
    \end{equation}
Write
    \begin{align*}
        H_d(N\Gamma_{g, 1}, N\Gamma_{g - 2, 1})
        &\defeq
        H_d(\Gamma(N_{g - 2, 1} \oplus S_{1, 1}), \Gamma(N_{g - 2, 1}))),
        \\
        H_d(\Gamma_{h, 1}, \Gamma_{h - 1, 1})
        &\defeq
        H_d(\Gamma(S_{h - 1, 1} \oplus S_{1, 1}), \Gamma(S_{h - 1, 1}))).
    \end{align*}

\begin{definition}\label{def:map1}
    The squares \cref{dia:stabilization_up_to_braiding} induce
        maps on relative group homology
        \[
            H_d(\Gamma(N_{g, 1} \oplus S_{1, 1}), \Gamma(N_{g, 1})))
            \to
            H_d(\Gamma(N_{g, 1} \oplus N_{1, 1} \oplus S_{1, 1}), \Gamma(N_{g, 1} \oplus N_{1, 1}))
        \]
        and, using the identification \cref{eqn:orientable_secondary_iso},
        \[
            H_d(\Gamma(S_{h, 1} \oplus S_{1, 1}), \Gamma(S_{h, 1})))
            \to
            H_d(\Gamma(N_{2h + 1, 1} \oplus S_{1, 1}), \Gamma(N_{2h + 1, 1})).
        \]
    Let $g \geq 4$.
    Then these maps induce the \emph{secondary stabilization map}
        for \cref{thm:secondary_stability},
        \[
            H_d(N\Gamma_{g - 1, 1}, N\Gamma_{g - 3, 1})
            \oplus
            H_d(\Gamma_{(g - 1)/2, 1}, \Gamma_{(g - 3)/2, 1})
            \to
            H_d(N\Gamma_{g, 1}, N\Gamma_{g - 2, 1}),
        \]
        when $g$ is odd, and
        \[
            H_d(N\Gamma_{g - 1, 1}, N\Gamma_{g - 3, 1})
            \to
            H_d(N\Gamma_{g, 1}, N\Gamma_{g - 2, 1})
        \]
        when $g$ is even.
\end{definition}

We now describe the map in \cref{thm:secondary_stability_2}.
For $g \geq 4$,
    the square
    \begin{equation}\label{dia:stabilization_up_to_braiding_2}
        \scalemath{0.86}{
        \begin{tikzcd}[column sep = large, ampersand replacement = \&]
            \Gamma(N_{g - 3, 1})
            \ar[rr, "\oplus \id_{S_{1, 1}}"]
            \ar[d, "\oplus \id_{N_{1, 1}}"]
                \&
                    \& \Gamma(N_{g - 3, 1} \oplus S_{1, 1})
                    \ar[d, "\oplus \id_{N_{1, 1}}"]
            \\
            \Gamma(N_{g - 3, 1} \oplus N_{1, 1})
            \ar[r, "\oplus \id_{S_{1, 1}}"]
                \& \Gamma(N_{g - 3, 1} \oplus N_{1, 1} \oplus S_{1, 1})
                \ar[r, "\Gamma(\id_S \oplus \beta^{-1})"]
                    \& \Gamma(N_{g - 3, 1} \oplus S_{1, 1} \oplus N_{1, 1})
        \end{tikzcd}
        }
    \end{equation}
    commutes by naturality of the braiding.
Analogously to \cref{eqn:orientable_secondary_iso},
    there are well-defined isomorphisms
    \begin{equation}\label{eqn:absorb_torus_hole}
        H_d(\Gamma(N_{g, 1} \oplus S_{1, 1} \oplus N_{1, 1}), \Gamma(N_{g, 1} \oplus S_{1, 1}))
        \cong
        H_d(N\Gamma_{g + 3, 1}, N\Gamma_{g + 2, 1}).
    \end{equation}

\begin{definition}\label{def:map2}
    Let $g \geq 4$ be even.
    The \emph{secondary stabilization map} for \cref{thm:secondary_stability_2}
        is the map
        \[
            H_d(N\Gamma_{g - 2, 1}, N\Gamma_{g - 3, 1})
            \to
            H_d(N\Gamma_{g, 1}, N\Gamma_{g - 1, 1}).
        \]
        induced by the squares \cref{dia:stabilization_up_to_braiding}
        using the identifications \cref{eqn:absorb_torus_hole}.
\end{definition}

\subsection{Organization of the paper}

The paper is structured as follows.
\begin{itemize}[(i)]
    \item[(1)]
    In \cref{sec:e2algebra},
        we define an $E_2$-algebra $R$ of mapping classes
        and prove that its $E_2$-homology vanishes in a range.
    In the process,
        we also prove that the complex of separating arcs on a surface is highly-connected,
        which may be of independent interest.
    \item[(2)]
    In \cref{sec:map},
        we construct a homotopy theoretic refinement of
            the secondary stabilization map of \cref{subsec:secondary_stabilization_map},
        the relative homology of which measures the failure of secondary stability.
    \item[(3)]
    In \cref{sec:calculation},
        we review calculations of the homology of mapping class groups of nonorientable surfaces
            due to Stukow, Paris, Szepietowski, and others.
    \item[(4)]
    In \cref{sec:secondary_stability},
        we construct a \q{presentation} of $R$
        and prove that it induces isomorphisms on the $E_2$-homology in a range of degrees.
    Using this, we leverage the Galatius--Kupers--Randal-Williams theory
        to prove our results.
\end{itemize}

\subsection{Acknowledgements}
The content of this paper was mostly extracted from my master's thesis.
I express my gratitude to my advisor, Professor Søren Galatius.
As pertains to this project,
    I am particularly grateful for his suggestion
        that I should consider the $E_2$-algebra of all surfaces,
    as opposed to restricting attention to only nonorientable surfaces.

I want to thank Luis Paris and B{\l}a{\.{z}}ej Szepietowski
    for thoroughly responding to my questions about their paper \cite{PS15}.

Finally, I want to thank
    Alexander Kupers for his comments and corrections
    and for advising me during my stay in Toronto.

% TODO: Check self-containment.

%! TEX root = main.tex

\section{The $E_2$-algebra of surface configurations}
\label{sec:e2algebra}

In this section,
    we study an $E_2$-algebra of mapping class groups of surfaces,
    culminating in a vanishing range for its $E_2$-homology.
For $0 < \epsilon < \frac12$,
    let $\del_\epsilon I^2 \subseteq I^2$ denote the open $\epsilon$-collar neighborhood,
    that is, the subset of points that have $< \epsilon$ Euclidean distance to a point in $\del I^2$.

\begin{definition}\label{def:mcg_cat}
    The braided strict monoidal groupoid $(\MCG, \oplus, (I^2 \times 0, 0), \beta)$ has
        \begin{alignat*}{3}
            \Ob(\MCG)
            \defeq{}&
            \{
                (C \subseteq I^2 \times \R^\infty, r \in \R_{>0})
                \mid{}&&
                \text{$C \cong S_{g, 1}$ or $N_{g, 1}$ for some $g > 1$,}
                \\
                &&&
                \del C = \del I^2 \times 0,
                \\
                &&&
                C - \del C \subset (0, 1)^2 \times \R^\infty,
                \\
                &&&
                \exists 0 < \epsilon < 1/2 :
                    \del_\epsilon I^2 \times 0
                    \subseteq C
            \}
            \\
            {}\cup{}&
            \{(I^2 \times 0, 0)\}
        \end{alignat*}
        as set of objects,
        and
        \begin{align*}
            \MCG((C_1, t_1), (C_2, t_2))
            &\defeq
            \frac{
                \{
                    f : C_1 \tox{\text{homeo}} C_2 \mid
                    f|\del I^2 \times 0 = \id
                \}
            }{
                \text{isotopy rel.~$\del I^2 \times 0$}
            }.
        \end{align*}
        as morphisms.
    In particular,
        $
            \Aut_\MCG((C, t)) = \Gamma(C).
        $
    The monoidal product $\oplus$ is the strictly associative boundary sum,
        given by horizontal scaling and horizontal juxtaposition
        \[
            (C_1, t_1) \oplus (C_2, t_2)
            \defeq
            \begin{cases}
                \(
                    \frac{t_1}{t_1 + t_2} \cdot C_1
                    \cup
                    \(\frac{t_1}{t_1 + t_2} + \frac{t_2}{t_1 + t_2} \cdot C_2\),
                    t_1 + t_2
                \),
                    & t_1 + t_2 > 0,
                \\
                (I^2 \times 0, 0),
                    & t_1 = t_2 = 0.
            \end{cases}
        \]
    \begin{figure}
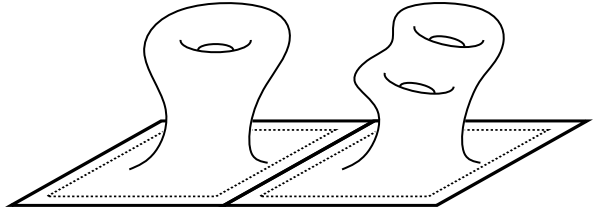

        % TODO: Use less lazy picture by changing the shape of the second one,
        %       maybe give it two bulbs.
        \incfig{juxtaposition}
        \caption{
            $(C_1, 1) \oplus (C_2, 1)$ for $C_1 \cong S_{1, 1}$ and $C_2 \cong S_{2, 1}$,
                where the $t$-values are visualized as width.
        }
    \end{figure}
    \noindent
    The braiding
        $
            \beta :
                (C_1, t_1) \oplus (C_2, t_2)
                \to
                (C_2, t_2) \oplus (C_1, t_1)
        $
        is the clockwise half Dehn twist.
    Up to a shrinking of the surfaces (expanding the flat collar inwards),
        this is the mapping class given by gluing $\id_{C_1}, \id_{C_2}$
        to the clockwise elementary braid $e$ of $S_{0, 3}$
            that swaps these boundaries;
        see \cref{fig:elementary_braid}.
    \begin{figure}[h]
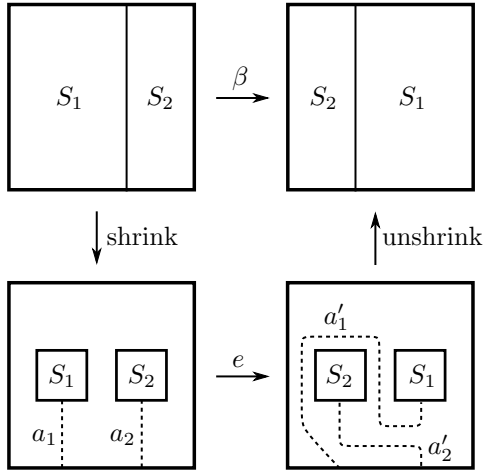

        \incfig{elementary_braid}
        \caption{
            Cutting out the surfaces $S_1$ and $S_2$,
                $e$ is the unique mapping class of a diffeomorphism of $S_{0, 3}$
                fixing the outer boundary,
                identically sending the left-hand inner boundary
                    to the right-hand inner boundary, and
                taking the isotopy class of the arc $a_1$ (respectively $a_2$)
                    to the isotopy class of the arc $a_1'$ (respectively $a_2'$).
        }
        \label{fig:elementary_braid}
    \end{figure}
\end{definition}

\begin{remark}\label{rem:big_cat}
    It is tempting to define $\Ob(\MCG)$ to be
        the set of diffeomorphism types of connected surfaces with one boundary component,
        so that each isomorphism class contains exactly one object.
    However, doing so, it becomes difficult, if not impossible,
        to extend the monoid structure of $\Ob(\MCG)$
        to a monoidal structure on $\MCG$.
    This has to do with the mixture of orientable and nonorientable surfaces.
    While there exist diffeomorphisms
        $
            N_{1, 1}^{\oplus g} \oplus S_{1, 1}^{\oplus h} \cong N_{1, 1}^{\oplus g + 2h},
        $
        there are no natural, canonical choices of such,
        and so it is unclear how one would coherently define the mapping class
        $
            f \oplus g \in \Gamma(N_{1, 1}^{\oplus g + 2h})
        $
        given
        $
            f \in \Gamma(N_{1, 1}^{\oplus g})
        $,
        $
            g \in \Gamma(S_{1, 1}^{\oplus h}).
        $
    On the other hand,
        if we considered only orientable or only nonorientable surfaces,
        it is easy to define a monoidal structure of boundary sums
            on the groupoid $\MCG'$
            whose objects
            $
                \Ob(\MCG') \defeq \{0, 1, 2, \ldots\}
            $
            are diffeomorphism types
                represented by genera.
\end{remark}

By a unital (respectively nonunital) $E_1$-algebra,
    we mean an algebra over the operad $S$ defined in \cref{sec:classifying}
        (respectively its nonunital version where $S(0)$ is replaced by $\emptyset$).
The corresponding free monad is denoted by $E_1^+$ (respectively $E_1$).
By a unital (respectively nonunital) $E_2$-algebra,
    we mean an algebra over the operad $B$ defined in \cref{sec:classifying}
        (respectively its nonunital version where $B(0)$ s replaced by $\emptyset$).
The corresponding free monad is denoted by $E_2^+$ (respectively $E_2$).
Note that these definitions differ from \cite{GKRW19Ek}
    which prefers the little $k$-cubes operads (see \cite[Definition 12.1]{GKRW19Ek}),
    but their formalism is equivalent to ours by \cref{prop:e2_models}.

\begin{definition}
    The unital $E_2$-algebra $R'$ in graded simplicial sets $\SSet^{\N_0}$ is
        the classifying space of
            the braided strict monoidal groupoid $(\MCG, \oplus, (I^2 \times 0, 0), \beta)$
        as constructed in \cref{sec:classifying} graded (see \cref{rem:graded_classifying})
            via the grading
            \begin{align*}
                h : \MCG &\to \N_0,
                \\
                (S, t) &\mapsto \begin{cases}
                    2g, & \text{if $S \cong S_{g, 1}$},
                    \\
                    g, & \text{if $S \cong N_{g, 1}$}.
                \end{cases}
            \end{align*}
    The nonunital $E_2$-algebra $R \in \Alg_{E_2}(\SSet^{\N_0})$ is obtained from $R'$ by
        restricting the $E_2^+$-algebra structure to an $E_2$-algebra structure and
        replacing $R'(0)$ with $\emptyset$.
\end{definition}

We write $H_{n, d}(R) \defeq H_d(R(n))$.

\begin{lemma}
    \[
        H_{n, d}(R)
        \cong
        \begin{cases}
            0,
                & \text{if $n = 0$},
            \\
            H_d(\Gamma_{n/2, 1}) \oplus H_d(N\Gamma_{n, 1}),
                & \text{if $n \geq 1$ and $n$ is even},
            \\
            H_d(N\Gamma_{n, 1}),
                & \text{if $n \geq 1$ and $n$ is odd}.
        \end{cases}
    \]
\end{lemma}

\begin{proof}
    This is because
        $
            \Aut_\MCG((C, t)) = \Gamma(C).
        $
\end{proof}

\subsection{Identification of the $E_1$-splitting complex}

The $E_2$-structure on $R$ forgets down to an $E_1$-structure.
Our next objective is to prove a vanishing range for the $E_1$-homology of $R$
    (see \cite[10.1.6]{GKRW19Ek}).

One of the key features of the Galatius--Kupers--Randal-Williams theory is
    the \emph{$E_1$-splitting complex} $S^{E_1}(g)$,
    which is a certain semisimplicial set associated to
        a monoidal groupoid $\mathcal G$ and
        an object $g \in \mathcal G$.
They prove that
    connectivity estimates for $S^{E_1}(g)$ for various $g \in \mathcal G$
    imply a vanishing estimate for the $E_1$-homology
        of an $E_1$-algebra associated to $\mathcal G$ (as in \cref{def:alg_model}).
Such an estimate, in turn, can often be \q{transferred up}
    to a vanishing estimate for the $E_2$-homology.
We refer the reader to \cite[17.2 and 14.2]{GKRW19Ek}
    for generalities on these constructions.

We will often notationally suppress
    length data when denoting objects of $\MCG$.

\begin{lemma}\label{lem:concat_inj}
    For $S_1, S_2 \in \MCG$,
        the map
        $
            \blank \oplus \blank :
                \Gamma(S_1) \times \Gamma(S_2)
                \to
                \Gamma(S_1 \oplus S_2)
        $
        is injective.
\end{lemma}

\begin{proof}
    Let $[a]$ be the isotopy class of the arc $a$ on $S_1 \oplus S_2$
        placed where the gluing occured.
    An application of the isotopy extension theorem shows
        that the stabilizer of $[a]$ under the action of $\Gamma(S)$ is $\Gamma(S - a)$.
    Therefore, $\blank \oplus \blank$ identifies with the inclusion
        $
            \St([a]) \incl \Gamma(S_1 \oplus S_2).
        $
\end{proof}

Associated to the monoidal category $(\MCG, \oplus)$
    are the $E_1$-splitting complexes
    $
        S^{E_1}(S) \in \SSSet
    $
    with
    \[
        S^{E_1}(S)_p
        \defeq
        \colim_{
            (S_0, \ldots, S_{p + 1}) \in \MCG^{p + 2}_{\neq U}
        }
        \MCG(S_0 \oplus \cdots \oplus S_{p + 1}, S)
    \]
    and the face map $d_i$ given by changing the subscript
    \[
        (S_0, \ldots, S_{p + 1})
        \mapsto
        (S_0, \ldots, S_i \oplus S_{i + 1}, \ldots, S_{p + 1}).
    \]
\cite[Assumption 17.1]{GKRW19Ek} is satisfied by definition, and
\cite[Assumption 17.2]{GKRW19Ek} is satisfied by our \cref{lem:concat_inj}.

\begin{definition}\label{def:separating_arcs}
    Let $S$ be an orientable or nonorientable, connected surface with nonempty boundary, and
    let $b_0, b_1$ be distinct points in $\del S$ on the same boundary component.
    The \emph{complex of separating arcs} $\mathscr D(S, b_0, b_1)$ is the simplicial complex
        whose $p$-simplices are collections of $p + 1$ distinct isotopy classes
            of arcs between $b_0, b_1$
            that admit representatives
            $
                a_0, \ldots, a_p
            $
            such that
            \begin{enumerate}
                \item
                for each $i \neq j$, $a_i \cap a_j = \{b_0, b_1\}$ and
                \item\label{def:separating_arcs:b}
                for each $i$, $S - a_i$ consists of two components,
                    none of which are diffeomorphic to $S_{0, 1}$.
            \end{enumerate}
\end{definition}

The complex of separating arcs is spiritually related to
    the complex of separating curves treated by Looijenga \cite{Looijenga13}.

\begin{figure}[h]
    \incfig{separating_1_simplex}
    \caption{
        A $1$-simplex
            $
                \<x, y\> \in \mathscr D(N_{4, 1}, b_0, b_1)_1.
            $
    }
\end{figure}

We shall temporarily abuse notation and write
    $b_0$ for the top left corner of the boundary and
    $b_1$ for the bottom left corner of the boundary
        for any object in $\MCG$.
Each monoidal product $S_0 \oplus \cdots \oplus S_{p + 1}$ admits
    a natural system
    $
        \sigma_{S_0 \oplus \cdots \oplus S_{p + 1}}
        \in \mathscr D(S_0 \oplus \cdots \oplus S_{p + 1}, b_0, b_1)_p
    $
    of separating arcs $a_0, \ldots, a_p$
    such that each $a_i$ is isotopic
        (allowing the endpoints to move along the horizontal boundary segments)
        to the arc where the $i$'th gluing occured.
Let $S \in \MCG$.
Endow $\mathscr D(S, b_0, b_1)$ with the structure of a semisimplicial set
    by ordering arcs from left to right at $b_0$.
For each $p \geq 0$ and product $S_0 \oplus \cdots \oplus S_{p + 1}$,
    there is a map
    \[
        \MCG(S_0 \oplus \cdots \oplus S_{p + 1}, S)
        \to
        \mathscr D(S, b_0, b_1)_p
    \]
    given by taking the image of $\sigma_{S_0 \oplus \cdots \oplus S_{p + 1}}$
        under diffeomorphisms.

\begin{lemma}\label{lem:splitting_cpx_id}
    The maps described above assemble into an isomorphism of semisimplicial sets
        \[
            \phi : S^{E_1}(S) \tox{\cong} \mathscr D(S, b_0, b_1).
        \]
\end{lemma}

\begin{proof}
    Clearly, the maps assemble into a map of semisimplicial sets $\phi$.
    A simplex $\sigma \in \mathscr D(S, b_0, b_1)_p$
        cuts $S$ into surfaces $S_0, \ldots, S_{p + 1}$,
        ordered from left to right.
    The diffeomorphism $S_0 \oplus \cdots \oplus S_{p + 1} \to S$,
        mending $S$ back together,
        determines an element $\tilde\sigma$ of $S^{E_1}(S)_p$
        such that $\phi(\tilde\sigma) = \sigma$.
    This shows surjectivity.
    For injectivity,
        suppose that $\phi(\overline \alpha) = \phi(\overline \beta)$
        for
        $
            \alpha \in \MCG(S_0 \oplus \cdots \oplus S_{p + 1}, S)
        $,
        $
            \beta \in \MCG(S_0' \oplus \cdots \oplus S_{p + 1}', S).
        $
    Since $\alpha$ and $\beta$ must preserve the ordering at $b_0$ and $b_1$, and
        since
        $
            \alpha(\sigma_{S_0 \oplus \cdots \oplus S_{p + 1}})
            =
            \beta(\sigma_{S_0' \oplus \cdots \oplus S_{p + 1}'})
        $,
        $\alpha$ and $\beta$ differ at most up to
            precomposition by diffeomorphisms of each summand,
        hence $\overline \alpha = \overline \beta$.
\end{proof}

\subsection{Auxiliary complexes}
\label{subsec:auxiliary}

Before we can execute our connectivity arguments,
    we must define some auxiliary simplicial complexes
    and review their connectivity estimates.
By convention,
    \q{$(-1)$-connected} means nonempty.

\begin{definition}[{\cite[\defn{3.1}]{Wahl07}}]
    Let $S$ be an orientable or nonorientable, connected surface with nonempty boundary, and
    let $b_0, b_1$ be distinct points in $\del S$.
    $BX(S, b_0, b_1)$ is the simplicial complex
        whose $p$-simplices are collections of $p + 1$ distinct isotopy classes
            of arcs between $b_0, b_1$
            that admit representatives
            $
                a_0, \ldots, a_p
            $
            such that
            \begin{enumerate}
                \item
                $a_i \cap a_j = \{b_0, b_1\}$ for each $i \neq j$ and
                \item
                $S - (a_0 \cup \cdots \cup a_p)$ is connected.
            \end{enumerate}
\end{definition}

For convenience,
    set
    \begin{align*}
        h(N_{g, r}) &\defeq g,
        \\
        h(S_{g, r}) &\defeq 2g.
    \end{align*}
This notation is compatible with the functor $h$ considered earlier.

\begin{theorem}[{\cite[\thm{3.2}]{Wahl07}}]\label{thm:bx_high_conn}
    Let $S$ be an orientable or nonorientable, connected surface with nonempty boundary,
    let $b_0, b_1$ be distinct points in $\del S$, and
    let $i$ be the number of boundary components of $S$
        that intersect $\{b_0, b_1\}$.
    Then $BX(S, b_0, b_1)$ is $(h(S) + i - 3)$-connected.
\end{theorem}

\begin{definition}[{\cite[\defn{5.3}]{Wahl07}}]
    Let $S$ be an orientable or nonorientable, connected surface.
    $\mathcal C_0(S)$ is the simplicial complex
        whose $p$-simplices are collections of $p + 1$ distinct isotopy classes
            of simple closed curves in $S$
            that admit representatives
            $
                c_0, \ldots, c_p
            $
            such that
            \begin{enumerate}
                \item
                $c_i \cap \del S = \emptyset$ for each $i$,
                \item
                $c_i \cap c_j = \emptyset$ for each $i \neq j$, and
                \item
                $S - (c_0 \cup \cdots \cup c_p)$ is connected.
            \end{enumerate}
\end{definition}

\begin{theorem}[{\cite[\thm{5.4}]{Wahl07}}]\label{thm:c0_high_conn}
    Let $S$ be an orientable or nonorientable, connected surface.
    Then $\mathcal C_0(S)$ is $\floor{\frac{h(S) - 3}{2}}$-connected.
\end{theorem}

\subsection{Connectivity estimates for the complex of separating arcs}

The objective of this subsection is to prove a connectivity estimate
    for the $E_1$-splitting complexes $S^{E_1}(S)$.
We use an induction argument that
    creates more boundary components in the induction step and
    will suggest the following auxiliary complex.

\begin{definition}
    Let $S$ be an orientable or nonorientable, connected surface
        with $r \geq 1$ boundary components,
    let $b_0, b_1 \in \del S$ be distinct points in the same boundary component, and
    let $\tilde b_0$ be an orientation of the boundary near $b_0$.
    $
        \mathscr D'(S, b_0, \tilde b_0, b_1) \subseteq \mathscr D(S, b_0, b_1)
    $
    is the subcomplex
        consisting of those collections $\sigma$ of isotopy classes of arcs
        that cut $S$ into surfaces $S_0, \ldots, S_{p + 1}$
            (ordered according to $\tilde b_0$ at $b_0$)
            such that $S_0, \ldots, S_p$ each have only one boundary component.
\end{definition}

We remind the reader of the function $h$ defined in \cref{subsec:auxiliary},
    which counts crosscaps by $1$ and torus holes by $2$.

\begin{theorem}\label{thm:biased_separating_arcs_high_conn}
    Let $(S, b_0, \tilde b_0, b_1)$ be as above.
    Then
    \begin{enumerate}
        \item\label{thm:biased_separating_arcs_high_conn:a}
        if $r = 1$,
            $\mathscr D'(S, b_0, \tilde b_0, b_1)$ is
                $\(\floor{\frac{h(S) - 1}{2}} - 2\)$-connected, or
        \item\label{thm:biased_separating_arcs_high_conn:b}
        if $r > 1$,
            $\mathscr D'(S, b_0, \tilde b_0, b_1)$ is
                $\(\floor{\frac{h(S) - 1}{2}} - 1\)$-connected.
    \end{enumerate}
\end{theorem}

The proof relies on the following theorem.
A poset $\mathcal X$ is said to be \emph{$n$-connected},
    if its nerve $N \mathcal X$ is $n$-connected.
A subposet $\mathcal Y \subseteq \mathcal X$ is said to be \emph{closed} if
    for any $y \in \mathcal Y$,
    $
        x < y
        \implies
        x \in \mathcal Y.
    $

\begin{theorem}[Nerve Theorem, {\cite[\cor{4.2}]{GKRW19MCG}}]\label{thm:nerve_thm}
    Let $\mathcal X$ be a poset,
    let $\mathcal A$ be another poset,
    let
        \[
            F : \mathcal A^\op \to \{\text{closed subposets of $\mathcal X$}\}
        \]
        be a map of posets,
    let
        $
            t_{\mathcal X} : \mathcal X \to \Z
        $
        and
        $
            t_{\mathcal A} : \mathcal A \to \Z
        $
        be functions of sets, and
    let $n \in \Z$.
    Suppose that
    \begin{enumerate}[(1)]
        \item\label{thm:nerve_thm:1}
        $\mathcal A$ is $(n - 1)$-connected,
        \item\label{thm:nerve_thm:2}
        for every $a \in \mathcal A$,
            $\mathcal A_{<a}$ is $(t_{\mathcal A}(a) - 2)$-connected and
            $F(a)$ is $(n - t_{\mathcal A}(a) - 1)$-connected, and
        \item\label{thm:nerve_thm:3}
        for every $x \in \mathcal X$,
            $\mathcal X_{<x}$ is $(t_{\mathcal X}(x) - 2)$-connected and
            the subposet
            \[
                \mathcal A_x \defeq \{a \in \mathcal A \mid x \in F(a)\}
            \]
            is $((n - 1) - t_{\mathcal X}(x) - 1)$-connected.
    \end{enumerate}
    Then $\mathcal X$ is $(n - 1)$-connected.
\end{theorem}

The following proof is inspired by the proof of \cite[\thm{4.9}]{GKRW19MCG}.
We use curves instead of arcs to ensure
    that cutting always decrease genus,
    as would not be the case when cutting along one-sided arcs.

\begin{proof}[Proof of \cref{thm:biased_separating_arcs_high_conn}]
    Ordering arcs according to $\tilde b_0$ at $b_0$
        endows $\mathscr D'(S, b_0, \tilde b_0, b_1)$
            with the structure of a semisimplicial set.
    Furthermore,
        it gives
        $
            \tilde{\mathscr D}'(S, b_0, \tilde b_0, b_1)
            \defeq
            \mathscr D'(S, b_0, \tilde b_0, b_1)_0
        $
        the structure of a poset.
    The canonical map of simplicial sets
        \[
            s_* \mathscr D'(S, b_0, \tilde b_0, b_1)
            \tox{\cong}
            N(\tilde{\mathscr D}'(S, b_0, \tilde b_0, b_1)),
        \]
        is an isomorphism,
        so it suffices to show
            that $\tilde{\mathscr D}'(S, b_0, \tilde b_0, b_1)$ is highly-connected.
    The cases
        \cref{thm:biased_separating_arcs_high_conn:a} and
        \cref{thm:biased_separating_arcs_high_conn:b}
        are proven simultaneously by induction on $h(S)$.
    The base case $h(S) = 0$ is vacuous,
        so suppose $h(S) \geq 1$.
    In each induction step,
        Case \cref{thm:biased_separating_arcs_high_conn:a} is proven
            before Case \cref{thm:biased_separating_arcs_high_conn:b},
            which uses Case \cref{thm:biased_separating_arcs_high_conn:a}.

    \para{Case \cref{thm:biased_separating_arcs_high_conn:a}}
    Suppose \cref{thm:biased_separating_arcs_high_conn:a}
        and \cref{thm:biased_separating_arcs_high_conn:b}
        are known for surfaces $S'$ with $h(S') < h(S)$.
    For a simplicial complex $X$, let $\f{simp}(X)$ denote the poset of simplices.
    Note that $\f{simp}(X)$ is $n$-connected
        if and only if $X$ is $n$-connected.

    Let
        $
            \mathcal A \defeq \f{simp}(\mathcal C_0(S)).
        $
    Consider the map of posets
        \begin{align*}
            F :
                \mathcal A^\op
                &\to
                \{\text{closed subposets of $\tilde{\mathscr D}'(S, b_0, \tilde b_0, b_1)$}\},
            \\
                \<c_0, \ldots, c_p\>
                &\mapsto
                \{a \mid \text{$c_0, \ldots, c_p$ are on the right of $a$}\}.
        \end{align*}
    By an isotopy class $c_i$ \q{being on the right} of an isotopy class $a$,
        we mean that there are representatives enjoying this property.
    Then we show \cref{thm:biased_separating_arcs_high_conn:a} by substituting
        \begin{align*}
            n &= \floor{\frac{h(S) - 1}{2}} - 1,
            \\
            t_{\mathcal A}\<a_0, \ldots, a_p\> &= p,
            \\
            t_{\mathcal X}(a) &= \floor{\frac{h(S_0 - 1)}{2}}
                \text{ for $S_0$ the surface left of $a$}
        \end{align*}
        into \cref{thm:nerve_thm}.
    We now check its assumptions.
    \begin{enumerate}
        \item[\cref{thm:nerve_thm:1}]
        $\mathcal A$ is $\floor{\frac{h(S) - 3}{2}}$-connected by \cref{thm:c0_high_conn},
            and so \cref{thm:nerve_thm:1} follows from the inequality
            \[
                \floor{\frac{h(S) - 1}{2}} - 2 \leq \floor{\frac{h(S) - 3}{2}}.
            \]

        \item[\cref{thm:nerve_thm:2}]
        For any $p$-simplex $\sigma \in \mathcal A$,
            $\mathcal A_{<\sigma}$ is isomorphic to $\del \Delta^p$,
            which is $(p - 2)$-connected,
            showing the first part of \cref{thm:nerve_thm:2}.
        On the other hand,
            the canonical map%
            \footnote{
                This map is well-defined by standard arguments.
                For example, one can check this
                    by an induction on intersections and \cite[\lem{3.2}]{Epstein1966}.
            }
            \[
                F(\sigma) \tox{\cong} \tilde{\mathscr D}'(S - \sigma, b_0, \tilde b_0, b_1)
            \]
            admits an inverse induced by the gluing map $S - \sigma \to S$.
        $S - \sigma$ has at least two boundary components and
            $
                h(S - \sigma) \geq h(S) - 2(p + 1).
            $
        Thus, by the induction hypothesis on \cref{thm:biased_separating_arcs_high_conn:b},
            $F(\sigma)$ is at least
            \[
                \floor{\frac{h(S) - 2(p + 1) - 1}{2}} - 1
                =
                \floor{\frac{h(S) - 1}{2}} - 2 - p
                =
                n - t_{\mathcal A}(\sigma) - 1
            \]
            -connected,
            showing the second part of \cref{thm:nerve_thm:2}.

        \item[\cref{thm:nerve_thm:3}]
        Let $x \in \tilde{\mathscr D}'(S, b_0, \tilde b_0, b_1)$.
        Suppose $x$ separates the surface into
            $S_0$ on the left and
            $S_1$ on the right.
        Then
            $
                \tilde{\mathscr D}'(S, b_0, \tilde b_0, b_1)_{<x}
                \cong
                \tilde{\mathscr D}'(S_0, b_0, \tilde b_0, b_1)
            $
            and
            $
                \mathcal A_x \cong \f{simp}(\mathcal C_0(S_1)).
            $
        By the induction hypothesis on
            \cref{thm:biased_separating_arcs_high_conn:a},
            the first part of assumption \cref{thm:nerve_thm:3} is clear.
        Note that $h(S_0) + h(S_1) = h(S)$,
            so
            \[
                \floor{\frac{h(S_0) - 1}{2}} + \floor{\frac{h(S_1) - 1}{2}}
                \geq
                \floor{\frac{h(S) - 1}{2}} - 1.
            \]
        By \cref{thm:c0_high_conn},
            $\mathcal C_0(S_1)$ is therefore
            \begin{align*}
                \floor{\frac{h(S_1) - 1}{2}} - 1
                &\geq
                \floor{\frac{h(S) - 1}{2}} - 1 - \floor{\frac{h(S_0) - 1}{2}} - 1
                \\
                &>
                (n - 1) - t_{\mathcal X}(x) - 1
            \end{align*}
            -connected.
        This shows the second part of \cref{thm:nerve_thm:3}.
    \end{enumerate}

    \para{Case \cref{thm:biased_separating_arcs_high_conn:b}}
    Suppose
        \begin{enumerate}[(i)]
            \item\label{item:induction_hypothesis:i}
            \cref{thm:biased_separating_arcs_high_conn:a} is known
                for $S'$ with $h(S') \leq h(S)$,
            \item\label{item:induction_hypothesis:ii}
            \cref{thm:biased_separating_arcs_high_conn:b} is known
                for $S'$ with $h(S') < h(S)$, and
            \item\label{item:induction_hypothesis:iii}
            \cref{thm:biased_separating_arcs_high_conn:b} is known
                for $S'$ with $2 \leq r' < r$ boundary components.
        \end{enumerate}
    Choose $b_2 \in \del S$ on the interior
        of the right-hand (according to $\tilde b_0$) segment between $b_0, b_1$ of the boundary,
        and choose $b_3 \in \del S$ on another boundary component.
    Let
        $
            \mathcal A \defeq \f{simp}(BX(S, b_2, b_3)).
        $
    Consider the map of posets
        \begin{align*}
            F :
                \mathcal A^\op
                &\to
                \{\text{closed subposets of $\tilde{\mathscr D}'(S, b_0, \tilde b_0, b_1)$}\},
            \\
                (a_0, \ldots, a_p)
                &\mapsto
                \{a \mid \text{$a_0, \ldots, a_p$ is on the right of $a$}\}.
        \end{align*}
    Then we show \cref{thm:biased_separating_arcs_high_conn:b} by substituting
        \begin{align*}
            n &= \floor{\frac{h(S) - 1}{2}},
            \\
            t_{\mathcal A}\<a_0, \ldots, a_p\> &= p,
            \\
            t_{\mathcal X}(a) &= \floor{\frac{h(S_0) - 1}{2}}
                \text{ for $S_0$ the surface left of $a$}
        \end{align*}
        into \cref{thm:nerve_thm}.
    We now check its assumptions.
    \begin{enumerate}
        \item[\cref{thm:nerve_thm:1}]
        $\mathcal A$ is $(h(S) - 2)$-connected by \cref{thm:bx_high_conn},
            and so \cref{thm:nerve_thm:1} follows from the inequality
            \[
                \floor{\frac{h(S) - 1}{2}} - 1 \leq h(S) - 1.
            \]

        \item[\cref{thm:nerve_thm:2}]
        For any $p$-simplex $\sigma \in \mathcal A$,
            $\mathcal A_{< \sigma}$ is isomorphic to $\del \Delta^p$,
            which is $(p - 2)$-connected,
            showing the first part of \cref{thm:nerve_thm:2}.
        If $S - \sigma$ has only one boundary component,
            $F(\sigma)$ is contractible
            as the arc parallel to
                the right-hand (according to $\tilde b_0$) boundary segment
                    between $b_0, b_1$ of $S - \sigma$
                constitutes a terminal element, and
        if $S - \sigma$ has multiple boundary components,
            $
                F(\sigma) \cong \tilde{\mathscr D}'(S - \sigma, b_0, \tilde b_0, b_1)
            $
            for reasons similar to those in \cref{thm:biased_separating_arcs_high_conn:a}.
        In the former case,
            there is nothing to show,
            so suppose that $S - \sigma$ has multiple boundary components.
        We have
            $
                h(S - \sigma) \geq h(S) - 2p.
            $
        Therefore,
            $F(\sigma)$ is at least
            \[
                \floor{\frac{h(S) - 2p - 1}{2}} - 1
                =
                \floor{\frac{h(S) - 1}{2}} - p - 1
            \]
            -connected
            by induction hypothesis
                \cref{item:induction_hypothesis:ii}
                    if $h(S - \sigma) < h(S)$ or
                \cref{item:induction_hypothesis:iii}
                    if $S - \sigma$ has fewer boundary components than $S$.
        This shows the second part of \cref{thm:nerve_thm:2}.

        \item[\cref{thm:nerve_thm:3}]
        Let $x \in \tilde{\mathscr D}'(S, b_0, \tilde b_0, b_1)$.
        Suppose $x$ separates the surface into
            $S_0$ on the left and
            $S_1$ on the right.
        Then
            $
                \tilde{\mathscr D}'(S, b_0, \tilde b_0, b_1)_{<x}
                \cong
                \tilde{\mathscr D}'(S_0, b_0, \tilde b_0, b_1)
            $
            and
            $
                \mathcal A_x \cong \f{simp}(BX(S_1, b_2, b_3)).
            $
        Note that $S_0$ has only one boundary component.
        By induction hypothesis
            \cref{item:induction_hypothesis:i},
            the first part of assumption \cref{thm:nerve_thm:3} is clear.
        Note that $h(S_0) + h(S_1) = h(S)$.
        By \cref{thm:bx_high_conn},
            $BX(S_1, b_2, b_3)$ is therefore
            \begin{align*}
                h(S_1) - 1
                &=
                h(S) - h(S_0) - 1
                \\
                &\geq
                (n - 1) - t_{\mathcal X}(x) - 1
            \end{align*}
            -connected.
        This shows the second part of \cref{thm:nerve_thm:3}.
    \end{enumerate}
\end{proof}

Since
    $
        \mathscr D(S, b_0, b_1)
        =
        \mathscr D'(S, b_0, \tilde b_0, b_1)
    $
    when $S$ has only one boundary component,
    \cref{thm:biased_separating_arcs_high_conn} will suffice for our purpose
        of giving a vanishing estimate for the $E_2$-homology of $R$.
However, the following result might be interesting in its own right.

\begin{theorem}\label{thm:separating_arcs_high_conn}
    Let $(S, b_0, b_1)$ be as above,
        where $S$ has $r \geq 1$ boundary components.
    Then the complex of separating arcs $\mathscr D(S, b_0, b_1)$ is
        $\(\floor{(h(S) - 1)/2} - 3 + r\)$-connected.
\end{theorem}

We prove this using \cref{thm:biased_separating_arcs_high_conn}
    and the technique of surgering arcs,
    which is originally due to Hatcher \cite{Hatcher91}.
We use a particular variant of this technique,
    which we learned from \cite{Wahl07}.

\begin{proof}
    \para{Case: $r = 1$}
    This is
        \cref{thm:biased_separating_arcs_high_conn}\cref{thm:biased_separating_arcs_high_conn:a}
        for $\tilde b_0$ is chosen in any way.

    \para{Case: $r > 1$}
    We apply a surgery argument to reduce to the case $r = 1$.
    Proceed by induction on $r$.
    Pick a boundary component $C$ distinct from the one containing $b_0, b_1$.
    We call a vertex $I$ of $\mathscr D(S, b_0, b_1)$ \emph{special}
        if it separates the surface $S$ into surfaces $S_0, S_1$
            one of which is the annulus around $C$.
    Let
        $
            \{I_j\} \subseteq \mathscr D(S, b_0, b_1)_0
        $
        be the set of special vertices.
    The case $(h(S), r) = (0, 2)$ is vacuous.
    Therefore, suppose either $h(S) > 0$ or $r > 2$
        so that $\{I_j\} \neq \emptyset$.
    We have
        \begin{equation}\label{eqn:special_decomposition}
            \mathscr D(S, b_0, b_1)
            =
            \mathscr D_\f{sp}(S, b_0, b_1)
                \bigcup_j \St(I_j),
        \end{equation}
        where
        $
            \mathscr D_\f{sp}(S, b_0, b_1)
            \subseteq
            \mathscr D(S, b_0, b_1)
        $
        is the subcomplex of simplices not containing special vertices, and
        where each star $\St(I_j)$ is attached to $\mathscr D_\f{sp}(S, b_0, b_1)$
            along the link of $I_j$.
    Indeed, two distinct stars have intersection
        contained in $\mathscr D_\f{sp}(S, b_0, b_1)$,
        since any two distinct special arcs intersect.
    Pick a special vertex $I_1 \in \{I_j\}$.

    \begin{claim}
        $
            X \defeq \mathscr D_\f{sp}(S, b_0, b_1) \cup \St(I_1)
        $
        is contractible.
    \end{claim}

    We say that an arc $a$ on $S$ is \emph{trivial}
        if $S - a$ has two components,
            one of which is diffeomorphic to $S_{0, 1}$.
    The definition of \q{special arcs} was effectively chosen
        so as to make the following argument go through.
    In particular, excluding special arcs ensures
        that the surgery does not create trivial arcs,
        as is illegal by \cref{def:separating_arcs}\cref{def:separating_arcs:b}.

    \begin{claimproof}
        We construct a contraction of $X$
            onto the star $\St(I_1)$,
            which is contractible.
        Fix a representative $I_1'$ for the isotopy class $I_1$.
        Let $I$ be a vertex of $X$.
        If $I \in \St(I_1)$,
            define $f(I) \defeq I$.
        If $I \not\in \St(I_1)$,
            any representative of $I$ intersects $I_1'$.
        Choose a representative $I'$ of $I$
            that intersects $I_1'$ transversely with minimal number of intersections.
        By standard arguments,
            such a choice is unique up to isotopy
                through arcs intersecting $I_1'$ transversely and minimally.
        Inside a fixed model of the $S_{0, 2}$ bounded by $I_1'$,
            label the segments of $I'$ by $\ell_1, \ldots, \ell_n$
                ordered starting with the one farthest from $b_0$.
        We arrange
            that $\ell_1, \ldots, \ell_{n - 1}$ are parallel straight lines through the annulus and,
            perhaps after relabeling $b_0, b_1$,
                the remaining $\ell_n$ either meets $b_1$ or is a straight line too.
        This places us in the situation of one of the following standard pictures.

        \begin{figure}
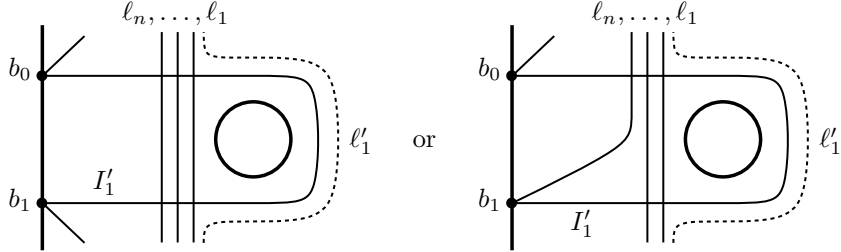

            \incfig{surgery_of_arcs_detour}
            \caption{
                The standard pictures.
            }
            \label{fig:detour}
        \end{figure}

        Replace the outermost segment $\ell_1$ of $I'$ with an arc $\ell_1'$
            that instead of intersecting $I_1'$ on the interior
                takes a detour around $I_1'$ to the other side.
        The resulting arc $\tilde I'$ obtained from making this replacement
            is again separating, nontrivial, and nonspecial,
            and it will have
                either $2n - 1 < 2n$ or $2n - 2 < 2n$ intersections with $I_1'$.
        Repeating this procedure,
            we eventually earn an arc $J'$,
            which has no intersections with $I_1'$.
        Define $f(I) \defeq J$
            for $J$ the isotopy class of $J'$.

        This defines a simplicial map
            $
                f :
                    X
                    \to
                    X
            $
            with image in $\St(I_1)$
            and we claim that
            $
                |f| \simeq \id_{|X|}.
            $
        If
            $
                \sigma = \<a_0, \ldots, a_p\>
                \in X
            $
            is not in $\St(I_1')$
            and $a_i$ contains the rightmost line through the annulus bounded by $I_1'$
                among all the lines spanned by $a_0, \ldots, a_p$ in the standard picture
                    as in \cref{fig:detour},
            then
            $
                \<a_0, \ldots, a_p, \tilde a_i\>
                \in X
            $
            also,
            where $\tilde a_i$ is $a_i$ after a single surgery step.
        By sliding along the latter simplex in a straight line to replace $a_i$ by $\tilde a_i$
            and iterating this procedure until all intersections with $I_1'$ are resolved,
            we obtain a well-defined homotopy
            $
                H_\sigma :
                    I \times |\sigma|
                    \to
                    |X|
            $
            between the geometric realizations of $f|\sigma$ and $\id_X|\sigma$.
        After possibly reparameterizing each $H_\sigma(\blank, x)$,%
            \footnote{
                For instance, proceed by induction over the skeleta of $X$.
            }
            all these homotopies patch together to a homotopy
            \[
                f \simeq \id_X :
                    I \times |X|
                    \to
                    |X|.
            \]
        Therefore, $f$ constitutes a homotopy equivalence.
    \end{claimproof}

    Note that for each $j$, $\Lk(I_j) \cong \mathscr D(S', b_0, b_1)$
        for $S' \defeq S \cup_C D^2$.
    With this, the claim, and the fact that stars are contractible,
        \cref{eqn:special_decomposition} implies that
        \[
            |\mathscr D(S, b_0, b_1)|
            \simeq
            \bigvee_{j \neq 1} S |\Lk(I_j)|
            \cong
            \bigvee_{j \neq 1} S |\mathscr D(S', b_0, b_1)|,
        \]
        where $S \blank$ denotes the unreduced suspension.
    The latter is $\(\floor{\frac{h(S) - 1}{2}} - 3 + r\)$-connected
        by the induction hypothesis on $r$.
    This completes the induction step.
\end{proof}

\subsection{Consequences for $E_2$-homology}

We now harvest vanishing ranges for $E_2$-homology
    from the connectivity estimates obtained above.

\begin{corollary}\label{cor:e1_vanishing_r}
    $
        H_{n, d}^{E_1}(R) = 0
    $
    for $d \leq \floor{\frac{n - 1}{2}} - 1$.
\end{corollary}

\begin{proof}
    By
        \cite[\lem{17.10}]{GKRW19Ek},
        \cref{thm:biased_separating_arcs_high_conn} (or \cref{thm:separating_arcs_high_conn}), and
        \cref{lem:splitting_cpx_id},
        $T^{E_1}(n)$ (see \cite[\defn{17.3}]{GKRW19Ek}) is
            $\floor{\frac{n - 1}{2}}$-connected.
    By \cite[\cor{17.4}]{GKRW19Ek}, and \cite[\thm{10.13}]{GKRW19Ek},
        $
            S^1 \smsh Q_\L^{E_1}(R)(n)
            =
            h_*(S^1 \smsh Q_\L^{E_1}(\underline *_{\neq U}))(n)
        $
        is $\floor{\frac{n - 1}{2}}$-connected as well.
    The difference in the definition of our $R$ from their $R$ is
        accounted for by \cref{lem:alg_model}.
    The statement follows.
\end{proof}

\begin{corollary}\label{cor:e2_vanishing_r}
    $
        H_{n, d}^{E_2}(R) = 0
    $
    for $d \leq \floor{\frac{n - 1}{2}} - 1$.
\end{corollary}

\begin{proof}
    Apply \cite[\thm{14.3}]{GKRW19Ek} with
        $l = 1$,
        $k = 2$, and
        $\rho(n) = \floor{\frac{n - 1}{2}} + 1$.
    The first assumption of said theorem follows from the inequality
        \[
            \floor{\frac{n_1 - 1}{2}} + 1
            +
            \floor{\frac{n_2 - 1}{2}} + 1
            \geq
            \floor{\frac{n_1 + n_2 - 1}{2}} + 1.
        \]
    The second assumption of said theorem follows from \cref{cor:e1_vanishing_r}.
    The statement follows.
\end{proof}

%! TEX root = main.tex

\section{Identification of the secondary stabilization map}
\label{sec:map}

The purpose of this section is to construct a graded simplicial set,
    the homology of which measures the failure of secondary stability.
This object will be the homotopy cofiber of
    a certain homotopy theoretic refinement of the secondary stabilization map
        of \cref{subsec:secondary_stabilization_map}.

We use a general setting.
Let $(\mathcal G, \oplus, 0)$ be a symmetric monoidal groupoid
    (for our purposes, $(\mathcal G, \oplus, 0) = (\N_0, +, 0)$ suffices),
let $(\mathcal S, \otimes, U)$ be a convenient symmetric monoidal, simplicial model category
    subject to the axioms of \cite[2.1 and 7.1]{GKRW19Ek}
    (for example, $(\mathcal S, \otimes, U) = (\SSet, \times, *)$), and
let
    $
        A \in \Alg_{E_2}(\mathcal S^{\mathcal G})
    $
    be an $E_2$-algebra.
Endow $\mathcal S^{\mathcal G}$ with the projective model structure
    and the Day convolution monoidal product.
For any $g \in \mathcal G$,
    define
    $
        U_g \defeq g_* U
        \in \mathcal S^{\mathcal G},
    $
    the object with
        $U_g(x) = \emptyset$ for $x \not\cong g$ and
        $U_g(x) = U$ for $x \cong g$.
For any
    two objects $|x|, |y| \in \mathcal G$,
    maps
    $
        x : U_{|x|} \to U^{E_2} A,
    $
    $
        y : U_{|y|} \to U^{E_2} A
    $
    in $\mathcal S^{\mathcal G}$,
    and
    a right $A$-module
    $
        M \in \Alg_{\otimes A}(\mathcal S^{\mathcal G})
    $
    with cofibrant and fibrant underlying object in $\mathcal S^{\mathcal G}$,
    we would like to construct a quotient $M \sslash (x, y)$
        that, morally, kills $x, y$ from the homology of $M$.
One general construction of this sort is given in \cite[12.2.3]{GKRW19Ek}.
Unfortunately, their construction is somewhat inexplicit,
    which in our application makes it difficult
        to obtain an explicit description of the secondary stabilization map
        as in \cref{subsec:secondary_stabilization_map}.
To overcome this issue,
    we give a very concrete, albeit less general, construction.
%We shall not make any attempts at showing
    %that our construction coincides with theirs.

%We emphasize that
    %if the reader is not concerned with
        %the description of \cref{subsec:secondary_stabilization_map} and
    %can contend with there existing \emph{some} secondary stabilization map
    %such that \cref{thm:secondary_stability} is true,
    %they can skip to \cref{subsec:secondary_stability_proof} and
    %substitute our construction $\blank \sslash (\tilde\sigma, \tilde\tau)$
        %with $\blank / (\tilde\sigma, \tilde\tau)$ of \cite[12.2.3]{GKRW19Ek}.
%For example,
    %this suffices for proving \cref{cor:torus_hole_stability}.

Since $A$ is an $E_2$-algebra,
    there is a homotopy coherent square in $\mathcal S^{\mathcal G}$,
    \begin{equation}\label{dia:htpy_com}
        \begin{tikzcd}
            U^A M \otimes U_{|x|} \otimes U_{|y|}
            \ar[r, "\cdot x"]
            \ar[d, "\cdot y"]
                & U^A M \otimes U_{|y|}
                \ar[d, "\cdot y"]
            \\
            U^A M \otimes U_{|x|}
            \ar[r, "\cdot x"]
                & U^A M,
        \end{tikzcd}
    \end{equation}
    where the structure homotopy is the clockwise elementary braid (half rotation) of $x$ and $y$,
        that is, it is induced via the $A$-module structure on $M$ by the composite
        \[
            \Delta^1 \otimes U_{|x|} \otimes U_{|y|}
            \tox{\alpha \otimes x \otimes y}
            B(2) \otimes U^{E_2} A \otimes U^{E_2} A
            \to
            U^{E_2} A,
        \]
        where
            $\alpha$ is the map $\Delta^1 \to N\Braid_2$
                classifying the elementary braid $e_1 \in B_2$, and
            the second map is the $B$-algebra structure map for $A$.

\begin{definition}
    Define
        $
            M \sslash y
            \defeq
            \hocofib(
                M \otimes U_{|y|}
                \to
                M
            )
            \in \Ho(\mathcal S_*^{\mathcal G}).
        $
    The homotopy coherent square \cref{dia:htpy_com} induce
        a (homotopy class of a) map of homotopy cofibers
        \[
            \cdot x :
                M \sslash y \otimes U_{|x|}
                \to
                M \sslash y
        \]
    Define
        $
            M \sslash (x, y)
            \in \Ho(\mathcal S_*^{\mathcal G})
        $
        to be the homotopy cofiber of this map.
    Since the homotopy coherent square \cref{dia:htpy_com} is functorially associated to $M$,
        this defines a functor
        \[
            \blank \sslash (x, y) :
                \Alg_{\otimes A}(\mathcal S^{\mathcal G})
                \to
                \Ho(\mathcal S_*^{\mathcal G}).
        \]
\end{definition}

Recall that we fixed models $N_{1, 1}, S_{1, 1} \in \MCG$
    in \cref{subsec:secondary_stabilization_map}.
For a surface $S \in \MCG$,
    let $\MCG_{\cong S}$ denote the subcategory of those objects isomorphic to $S$.
A homotopy coherent square $\square$ \emph{rectifies} to a strict square $\square'$
    if there is a zig-zag of weak equivalences of homotopy coherent squares
        (that is, pointwise weak equivalences which respect the structure homotopies)
        from $\square$ to $\square'$.

\begin{proposition}\label{prop:rectification}
    Let $S \in \MCG$.
    Then the homotopy coherent square
        \begin{equation}\label{dia:conf_clockwise_braid}
            \begin{tikzcd}[column sep = huge]
                N\MCG_{\cong S}
                \ar[r, "\oplus N_{1, 1}"]
                \ar[d, "\oplus S_{1, 1}"]
                    & N\MCG_{\cong S \oplus N_{1, 1}}
                    \ar[d, "\oplus S_{1, 1}"]
                \\
                N\MCG_{\cong S \oplus S_{1, 1}}
                \ar[r, "\oplus N_{1, 1}"]
                    & N\MCG_{\cong S \oplus N_{1, 1} \oplus S_{1, 1}}
            \end{tikzcd}
        \end{equation}
        in $\SSet$
        with structure homotopy induced by the braiding natural isomorphism
        rectifies to the nerve functor $N$ applied to
            the strict square \cref{dia:stabilization_up_to_braiding}.
\end{proposition}

\begin{proof}
    Consider the zig-zag of homotopy coherent squares,
        \vspace{-1.5em}
        \begin{center}
            \begin{tikzpicture}
                \node (1) at (0, 0) {\cref{dia:conf_clockwise_braid}};
                \node (2) at (0, -2) {$\left[
                    \begin{tikzcd}[column sep = huge, ampersand replacement = \&]
                        N \Aut_\MCG(S)
                        \ar[rd, phantom, "\squared{1}"]
                        \ar[r]
                        \ar[d]
                            \& N \Aut_\MCG(S \oplus N_{1, 1})
                            \ar[d]
                        \\
                        N \Aut_\MCG(S \oplus S_{1, 1})
                        \ar[r]
                            \& N \MCG\<S \oplus N_{1, 1} \oplus S_{1, 1}, S \oplus S_{1, 1} \oplus N_{1, 1}\>
                    \end{tikzcd}
                \right]$};
                \node (3) at (0, -5) {$\left[
                    \begin{tikzcd}[column sep = huge, ampersand replacement = \&]
                        N \Aut_\MCG(S)
                        \ar[rd, phantom, "\squared{2}"]
                        \ar[r]
                        \ar[d]
                            \& N \Aut_\MCG(S \oplus N_{1, 1})
                            \ar[d]
                        \\
                        N \Aut_\MCG(S \oplus S_{1, 1})
                        \ar[r]
                            \& N \Aut_\MCG(S \oplus N_{1, 1} \oplus S_{1, 1})
                    \end{tikzcd}
                \right].$};
                \draw[->] (2)--(1) node [midway, left] {$\simeq$}
                                   node [midway, right] {$(1)$};
                \draw[->] (2)--(3) node [midway, left] {$\simeq$}
                                   node [midway, right] {$(2)$};
            \end{tikzpicture}
        \end{center}
    Here, $\MCG\<S \oplus N_{1, 1} \oplus S_{1, 1}, S \oplus S_{1, 1} \oplus N_{1, 1}\>$
        denotes the full subgroupoid of $\MCG$ generated by
        $
            S \oplus N_{1, 1} \oplus S_{1, 1}
        $
        and
        $
            S \oplus S_{1, 1} \oplus N_{1, 1}.
        $
    $(1)$ is induced by essentially surjective inclusions of full subgroupoids
        hence is a weak equivalence.
    $\squared{1}$ has as structure homotopy
        the restriction of \cref{dia:conf_clockwise_braid}'s structure homotopy.
    $\squared{2}$ is $N$ applied to the strict square \cref{dia:stabilization_up_to_braiding}.
    $(2)$ retracts the lower right groupoid
        $
            \MCG\<S \oplus N_{1, 1} \oplus S_{1, 1}, S \oplus S_{1, 1} \oplus N_{1, 1}\>
        $
        to
        its full subgroupoid on the object $S \oplus N_{1, 1} \oplus S_{1, 1}$,
        mapping $S \oplus N_{1, 1} \oplus S_{1, 1}$ and automorphisms thereof identically, and
        mapping $S \oplus S_{1, 1} \oplus N_{1, 1}$ to $S \oplus N_{1, 1} \oplus S_{1, 1}$ and
            automorphisms of $S \oplus S_{1, 1} \oplus N_{1, 1}$ to
            automorphisms of $S \oplus N_{1, 1} \oplus S_{1, 1}$ by means of the braiding $\beta$.
\end{proof}

Set $M \defeq R$ and
    fix maps
    \begin{align*}
        x \defeq \tilde\sigma :
            U_1
            &\to
            R,
        \\
        y \defeq \tilde\tau :
            U_2
            &\to
            R
    \end{align*}
    classifying the fixed models $N_{1, 1}, S_{1, 1} \in \MCG$ with length parameter $1$.
In this setting, the homotopy coherent square \cref{dia:htpy_com} takes the form
    \begin{equation}\label{dia:htpy_com_r}
        \begin{tikzcd}
            R \otimes U_1 \otimes U_2
            \ar[r, "\cdot \tilde\sigma"]
            \ar[d, "\cdot \tilde\tau"]
                & R \otimes U_2
                \ar[d, "\cdot \tilde\tau"]
            \\
            R \otimes U_1
            \ar[r, "\cdot \tilde\sigma"]
                & R
        \end{tikzcd}
    \end{equation}
    in $\SSet^{\N_0}$.

\begin{lemma}\label{lem:secondary_stabilization_map_id}
    Let $g \geq 4$.
    Then the map
        \[
            \blank \cdot \tilde\sigma :
                \tilde H_{g, d}(R \sslash \tilde\tau \otimes U_1)
                \to
                \tilde H_{g, d}(R \sslash \tilde\tau)
        \]
        induced by \cref{dia:htpy_com_r}
        identifies with the secondary stabilization map
            (cf.~\cref{def:map1})
        \[
            H_d(N\Gamma_{g - 1, 1}, N\Gamma_{g - 3, 1})
            \oplus
            H_d(\Gamma_{(g - 1)/2, 1}, \Gamma_{(g - 3)/2, 1})
            \to
            H_d(N\Gamma_{g, 1}, N\Gamma_{g - 2, 1})
        \]
        when $g$ is odd, and
        the composite of the secondary stabilization map and the inclusion
        \begin{align*}
            H_d(N\Gamma_{g - 1, 1}, N\Gamma_{g - 3, 1})
            &\to
            H_d(N\Gamma_{g, 1}, N\Gamma_{g - 2, 1})
            \\
            &\to
            H_d(N\Gamma_{g, 1}, N\Gamma_{g - 2, 1})
            \oplus
            H_d(\Gamma_{g/2, 1}, \Gamma_{g/2 - 1, 1})
        \end{align*}
        when $g$ is even.
\end{lemma}

\begin{proof}
    Evaluated on an odd $g \geq 4$,
        \cref{dia:htpy_com_r} is
        \[
            \begin{tikzcd}
                N\MCG_{\cong N_{g - 3, 1}}
                \sqcup
                N\MCG_{\cong S_{(g - 3)/2, 1}}
                \ar[r, "\oplus N_{1, 1}"]
                \ar[d, "\oplus S_{1, 1}"]
                    & N\MCG_{\cong N_{g - 2, 1}}
                    \ar[d, "\oplus S_{1, 1}"]
                \\
                N\MCG_{\cong N_{g - 1, 1}}
                \sqcup
                N\MCG_{\cong S_{(g - 1)/2, 1}}
                \ar[r, "\oplus N_{1, 1}"]
                    & N\MCG_{\cong N_{g, 1}},
            \end{tikzcd}
        \]
        where the structure homotopy is given by the braiding.
    Evaluated on an even $g \geq 4$,
        \cref{dia:htpy_com_r} is
        \[
            \left[
                \begin{tikzcd}
                    N\MCG_{\cong N_{g - 3, 1}}
                    \ar[r, "\oplus N_{1, 1}"]
                    \ar[d, "\oplus S_{1, 1}"]
                        & N\MCG_{\cong N_{g - 2, 1}}
                        \ar[d, "\oplus S_{1, 1}"]
                    \\
                    N\MCG_{\cong N_{g - 1, 1}}
                    \ar[r, "\oplus N_{1, 1}"]
                        & N\MCG_{\cong N_{g, 1}}
                \end{tikzcd}
            \right]
            \sqcup
            \left[
                \begin{tikzcd}[column sep = tiny, ampersand replacement = \&]
                    \emptyset
                    \ar[r]
                    \ar[d]
                        \& N\MCG_{\cong S_{g/2 - 1, 1}}
                        \ar[d, "\oplus S_{1, 1}"]
                    \\
                    \emptyset
                    \ar[r]
                        \& N\MCG_{\cong S_{g/2, 1}}
                \end{tikzcd}
            \right],
        \]
        where the structure homotopy of the first summand
            is again given by the clockwise half Dehn twist.
    The statement now follows from \cref{prop:rectification}.
\end{proof}

\begin{lemma}\label{lem:secondary_stabilization_map_id_2}
    Let $g \geq 4$ be an even integer.
    Then the map
        \[
            \blank \cdot \tilde\tau :
                \tilde H_{g, d}(R \sslash \tilde\sigma \otimes U_2)
                \to
                \tilde H_{g, d}(R \sslash \tilde\sigma)
        \]
        induced by \cref{dia:htpy_com_r}
        identifies with the map
        \[
            H_d(N\Gamma_{g - 2, 1}, N\Gamma_{g - 3, 1})
            \oplus
            H_d(\Gamma_{(g - 2)/2, 1})
            \to
            H_d(N\Gamma_{g, 1}, N\Gamma_{g - 1, 1})
            \oplus
            H_d(\Gamma_{g/2, 1})
        \]
        induced by
            the secondary stabilization map (cf.~\cref{def:map2}) and
            the torus hole stabilization map (cf.~\cref{sec:introduction}).
\end{lemma}

\begin{proof}
    This follows from the proof of \cref{lem:secondary_stabilization_map_id} using that
        the square \cref{dia:stabilization_up_to_braiding} for $S = N_{g - 3, 1}$
        is isomorphic to
        the square \cref{dia:stabilization_up_to_braiding_2} flipped along the diagonal.
\end{proof}

%! TEX root = main.tex

\section{Low-dimensional calculations}
\label{sec:calculation}

\subsection{Path components}

Let $FX$ denote the free commutative semigroup on a set $X$.

\begin{proposition}\label{prop:pi_0_r}
    The natural map of commutative semigroups
        \[
            F\{S_{1, 1}, N_{1, 1}\}
                / (S_{1, 1} \oplus N_{1, 1} = N_{1, 1}^{\oplus 3})
            \to
            \pi_{*, 0}(R)
        \]
        is an isomorphism.
\end{proposition}

\begin{proof}
    This follows from the construction of $R$
        using that Euler characteristic is additive with respect to boundary sum.
\end{proof}

\subsection{Review of known unstable homology groups of $N\Gamma_{g, r}$}
\label{subsec:homology_calculation}

Aside from the stable homology groups of $N\Gamma_{g, 1}$,
    some of which were computed by Randal-Williams \cite{RandalWilliams08},%
    \footnote{
        Randal-Williams considers $N\Gamma_{\infty, 0}$,
            but this has the same homology as $N\Gamma_{\infty, 1}$ by \cite[\thm{A(3)}]{Wahl07}.
    }
    not many homology groups beyond $H_1$ are known.
This subsection is a review of these unstable calculations.

\begin{proposition}\label{prop:n_gamma_1_1}
    $N\Gamma_{1, 1} = 0$.
\end{proposition}

\begin{proof}
    Consider a parameterized one-sided arc
        $
            c : I \to N_{1, 1}
        $
        from some $b_0 \in \del N_{1, 1}$ to itself.
    Since its complement is $S_{0, 1}$ and $\Gamma_{0, 1} = 0$
        (the Alexander trick \cite[\lem{2.1}]{FM11}),
        the isotopy class of a boundary-fixing diffeomorphism
        $
            \phi : N_{1, 1} \to N_{1, 1}
        $
        is determined by the isotopy class of $\phi c$.
    Since the boundary is fixed,
        $\phi$ induces the identity
        $
            \Z \to \Z
        $
        on $\pi_1$
        (it must send $2 \mapsto 2$).
    Therefore, $c \simeq \phi c$,
        which implies that $c$ and $\phi c$ are isotopic by \cite[3.1]{Epstein1966},
        and thus $\phi$ is isotopic to the identity.
\end{proof}

\begin{definition}
    Fix a model $N_{1, 1} \in \MCG$ for the Möbius strip.
    Write $N_{g, 1} \defeq N_{1, 1}^{\oplus g}$.
    The \emph{crosscap transposition}
        of the the pair of crosscaps $(i, i + 1)$ on $N_{g, 1}$
        is the mapping class
        \[
            \id_{N_{i, 1}} \oplus \beta \oplus \id_{N_{g - i - 2, 1}} :
                N_{i, 1} \oplus N_{1, 1}^{\oplus 2} \oplus N_{g - i - 2, 1}
                \tox{\cong}
                N_{i, 1} \oplus N_{1, 1}^{\oplus 2} \oplus N_{g - i - 2, 1},
        \]
        where
        $
            \beta :
                N_{1, 1}^{\oplus 2}
                \to
                N_{1, 1}^{\oplus 2}
        $
        is the unique (by \cref{prop:n_gamma_1_1}) mapping class
            sending $a_1$ to $a_2$ up to isotopy:
        \begin{figure}
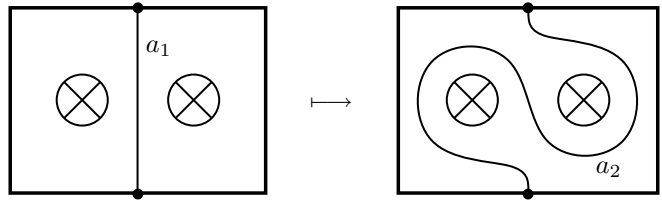

            \incfig{crosscap_transposition}
            \caption{
                $
                    \beta : [a_1] \mapsto [a_2].
                $
            }
            \label{fig:two_one_sided_arcs}
        \end{figure}
\end{definition}

\begin{proposition}\label{prop:homology_n_gamma_2_1}
    \[
        H_*(N\Gamma_{2, 1})
        =
        \begin{cases}
            \Z,
                & \text{if $* = 0$}, \\
            Z_2 b \oplus \Z u,
                & \text{if $* = 1$}, \\
            0,
                & \text{if $* > 1$},
        \end{cases}
    \]
    where
        $b$ is a Dehn twist along a nonseparating, two-sided curve and
        $u$ is the crosscap transposition.
\end{proposition}

\begin{proof}
    Stukow \cite[A.2]{Stukow06} showed that
        \[
            N\Gamma_{2, 1} = \<b, y \mid byb = y\> = \<b, u \mid bub = u\>,
        \]
        where $y \defeq bu$ is the \emph{crosscap slide}.
    The statement is now immediate for $* = 0, 1$.
    For $* > 1$,
        observe that the Klein bottle $K = N_{2, 0}$
            is aspherical and
            has
            $
                \pi_1(K) \cong N\Gamma_{2, 1}
            $
        hence is a $K(N\Gamma_{2, 1}, 1)$.
    However, $H_*(K) = 0$ for $* > 1$.
\end{proof}

\begin{proposition}\label{prop:homology_n_gamma_g_1}
    Let $a$ denote a Dehn twist along a nonseparating, two-sided curve
        with nonorientable complement,
    let $b$ denote a Dehn twist along a nonseparating, two-sided curve
        with orientable complement, and
    let $u$ denote a crosscap transposition.
    Then
    \begin{enumerate}
        \item
        $H_1(N\Gamma_{3, 1}) = Z_2 a \oplus Z_2 u$,
        \item
        $H_1(N\Gamma_{4, 1}) = Z_2 a \oplus Z_2 b \oplus Z_2 u$,
        \item
        $H_1(N\Gamma_{g, 1}) = Z_2 a \oplus Z_2 u$
            for $g \in \{5, 6\}$, and
        \item
        $H_1(N\Gamma_{g, 1}) = Z_2 u$
            for $g \geq 7$.
    \end{enumerate}
\end{proposition}

\begin{proof}
    This is a matter of abelianizing the presentation
        obtained by Paris and Szepietowski \cite[\thm{3.5}]{PS15}
        combined with the $H_1$ calculations of $\Gamma_{g, 1}$ and $\Gamma_{g, 2}$
            in \cite{Korkmaz02}.
    We omit these straight-forward calculations.
\end{proof}

% TODO: Remove this remark in the version of the thesis you turn in.
\begin{remark}
    The mapping classes $a, u$ are uniquely determined up to conjugation
        hence determine unique classes in $H_1$,
        whereas the conjugacy class of $b$ is only determined up to inverse,
            since its complement is orientable
            and so one cannot construct a diffeomorphism reversing the curve.
    However, $b$ is nonetheless uniquely determined in $H_1$,
        as $2b = 0$ in $H_1$ per the statements,
        so the ambiguity up to a sign disappears.
\end{remark}

\begin{remark}\label{rem:factor_stab_map}
    The torus hole stabilization map
        \[
            H_1(N\Gamma_{2, 1}) \to H_1(N\Gamma_{4, 1})
        \]
        takes $b$ to $b$,
        whereas the composite of crosscap stabilization maps
        \[
            H_1(N\Gamma_{2, 1}) \to H_1(N\Gamma_{3, 1}) \to H_1(N\Gamma_{4, 1})
        \]
        takes $b$ to $a$.
    This shows that the torus hole stabilization map
        does not always factor as two crosscap stabilization maps.
\end{remark}

\subsection{Failure of stability on $H_2$}
We make the following observation
    for the purpose of compiling the tables in \cref{sec:introduction}.

\begin{proposition}
    None of
        \begin{enumerate}
            \item
            $H_2(N\Gamma_{3, 1}, N\Gamma_{2, 1})$,
            \item
            $H_2(N\Gamma_{5, 1}, N\Gamma_{4, 1})$,
            \item
            $H_2(N\Gamma_{7, 1}, N\Gamma_{6, 1})$,
            \item
            $H_2(N\Gamma_{4, 1}, N\Gamma_{2, 1})$,
            \item
            $H_2(N\Gamma_{6, 1}, N\Gamma_{4, 1})$,
            \item
            $H_2(N\Gamma_{7, 1}, N\Gamma_{5, 1})$, or
            \item
            $H_2(N\Gamma_{8, 1}, N\Gamma_{6, 1})$
        \end{enumerate}
        vanish.
\end{proposition}

\begin{proof}
    This is immediate from
        the associated long exact sequences and
        the results of \cref{subsec:homology_calculation}.
\end{proof}

%! TEX root = main.tex

\section{A presentation for the $E_2$-algebra $R$}
\label{sec:secondary_stability}

In this section, we prove \cref{thm:secondary_stability}
    using a strategy similar to the one used
    to prove the generic stability theorem \cite[\thm{18.1}]{GKRW19Ek}.

As we are interested in the homology of $R$,
    we may as well linearize and consider instead
    the $E_2$-algebra
    $
        R_\Z \in \Alg_{E_2}(\sMod_\Z^{\N_0})
    $
    in graded simplicial $\Z$-modules.
We write
    $
        H_{n, d}(R_\Z) \defeq \pi_d(R_\Z(n))
    $
    and similar for other simplicial modules,
    reflecting the Dold--Kan correspondence,
    which then says that
    \[
        H_{n, d}(R_\Z) \cong H_{n, d}(R).
    \]

To simplify the argument,
    we are going to make use of the operations
    \[
        Q^1_\Z :
            H_{n, 0}(A)
            \to
            H_{2n, 1}(A)
    \]
    for
    $
        A \in \Alg_{E_2}(\sMod_\Z).
    $
These were described in \cite[3.1]{GKRW19MCG}
    and have two important properties,
    \begin{enumerate}[(1)]
        \item\label{item:integral_refinement:1}
        $Q^1_\Z(x) \otimes \F_2 = Q^1_{\F_2}(x \otimes \F_2)$,
            where
            $
                Q^1_{\F_2} \defeq \xi :
                    H_{n, 0}(A \otimes \F_2)
                    \to
                    H_{2n, 1}(A \otimes \F_2)
            $
            is the Dyer--Lashof top operation, and
        \item\label{item:integral_refinement:2}
        $
            2 Q^1_\Z(x) = -[x, x].
        $
    \end{enumerate}
Their purpose is to unify arguments that would otherwise need to be repeated
    first over $\F_2$ and then over $\F_\p$ for $\p$ an odd prime.

Recall the homology classes $d, b, u$ from \cref{sec:calculation}.
Let
    $
        \sigma \in H_{1, 0}(R_\Z)
    $
    denote the generator corresponding to the path-component for $N_{1, 1}$, and
let
    $
        \tau \in H_{2, 0}(R_\Z)
    $
    denote the generator corresponding to the path-component for $S_{1, 1}$.
Let $\RO \subseteq R$ be the sub-$E_2$-algebra on the orientable surfaces, and
let
    $
        d \in H_{2, 1}((\RO)_\Z) \cong H_1(\Gamma_{1, 1})
    $
    denote a class represented by a Dehn twist along a nonseparating curve.
(There are two choices for $d$,
    one for each orientation.)
Since Dehn twists along nonseparating curves generate the mapping class group $\Gamma_{g, 1}$
    (see \cite[\thm{5.4}]{FM11}),
    $d$ generates $H_{2, 1}((\RO)_\Z)$ and
    $\tau \cdot d$ generates $H_{4, 1}((\RO)_\Z)$.
Pick $n \in \Z$
    such that
    \[
        n \cdot \tau \cdot d
        =
        Q^1_\Z(\tau)
        \in H_{4, 1}((\RO)_\Z) \subseteq H_{4, 1}(R_\Z).
    \]
Fix representive maps
    $
        \tilde d, \tilde b, \tilde u
    $
    from spheres to $R_\Z$
    representing $d, b, u$ respectively.
Also, let $\tilde\sigma, \tilde\tau$ be the maps
    classifying the models $N_{1, 1}, S_{1, 1} \in \MCG$
        that we fixed in \cref{subsec:secondary_stabilization_map}.
Using these,
    define a map of nonunital $E_2$-algebras
    \[
        A'
        \defeq
        E_2(
            S_\Z^{1, 0} \sigma
            \oplus S_\Z^{2, 0} \tau
            \oplus S_\Z^{2, 1} d
            \oplus S_\Z^{2, 1} b
            \oplus S_\Z^{2, 1} u
        )
        \to
        R_\Z.
    \]
% TODO: We haven't actually defined $Q_\Z^1(\tilde\tau)$ for $\tilde\tau$ a strict map.
Fix a map
    $
        \tilde Q_\Z^1(\tau) :
            S^{4, 1}
            \to
            A'
    $
    representing
    $
        Q_\Z^1(\tau) \in H_{4, 1}(A')
    $
    and let also
    $
        \tilde Q_\Z^1(\tau) :
            S^{4, 1}
            \to
            A'
            \to
            R_\Z
    $
    denote the composite
    representing
    $
        Q_\Z^1(\tau) \in H_{4, 1}(R).
    $
By choosing nullhomotopies of
    $\tilde\sigma \tilde\tau - \tilde\sigma^3$ and
    $\tilde Q_\Z^1(\tau) - n \cdot \tilde\tau \cdot \tilde d$
    witnessing their triviality in $H_{*, *}(R_\Z)$,
    we pick an extension
    \begin{align*}
        A
        \defeq&\,
        E_2(
            S_\Z^{1, 0} \sigma
            \oplus S_\Z^{2, 0} \tau
            \oplus S_\Z^{2, 1} d
            \oplus S_\Z^{2, 1} b
            \oplus S_\Z^{2, 1} u
        )
        \\
        &
        \cup^{E_2}_{\sigma \tau - \sigma^3}
            D_\Z^{3, 1} \lambda
        \cup^{E_2}_{\tilde Q_\Z^1(\tau) - n \cdot \tau \cdot d}
            D_\Z^{4, 2} \rho
        \tox{f}
        R_\Z
    \end{align*}
    of $A' \to R_\Z$.
This is the presentation of $R_\Z$
    we shall use to prove \cref{thm:secondary_stability}.
Here, for a ring $B$,
    $S_B^{n, p} \in \sMod_B^{\N_0}$ is defined by
    \[
        S_B^{n, p}(g)
        \defeq
        \begin{cases}
            B[\Delta^p] / B[\del\Delta^p],
                & \text{if $g = n$},
            \\
            0,
                & \text{else},
        \end{cases}
    \]
    (as opposed to just linearizing the sphere in $\SSet$!)
    and $D_B^{n, p} \in \sMod_B^{\N_0}$ is defined by
    \[
        D_B^{n, p}(g)
        \defeq
        \begin{cases}
            B[\Delta^p],
                & \text{if $g = n$},
            \\
            0,
                & \text{else}.
        \end{cases}
    \]
$\blank \cup^{E_2}_x D_B^{n, p}$ denotes $E_2$-cell attachment;
    see \cite[6.1.1]{GKRW19Ek}.

\begin{figure}[h]
    \bgroup
    \def\arraystretch{2}
    \setlength{\tabcolsep}{1.8em}
    \begin{tabular}{c | c c c c c c c c c}
        $2$
        &
        &
        &
        & $\rho$
        \\
        $1$
        &
        & $d, b, u$
        & $\lambda$
        &
        \\
        $0$
        & $\sigma$
        & $\tau$
        &
        &
        \\
        \hline
        $\rfrac{d}{n}$ &
        $1$ &
        $2$ &
        $3$ &
        $4$
    \end{tabular}
    \egroup
    \caption{
        Bidegrees $(n, d)$ of generators and relators in $A$.
    }
\end{figure}

\begin{lemma}\label{lem:h0_a}
    The map
        $
            H_{*, 0}(A) \to H_{*, 0}(R_\Z)
        $
        is an isomorphism of $\N_0$-graded nonunital rings.
\end{lemma}

\begin{proof}
    Let
        $
            \f{Rng} : \Ab^{\N_0} \to \Ab^{\N_0}
        $
        be the graded, commutative, nonunital ring monad.
    $
        H_{*, 0} :
            \Alg_{E_2}(\sMod_\Z^{\N_0})
            \to
            \Alg_\f{Rng}(\Ab^{\N_0})
    $
    preserves colimits.
    Since
        $
            H_{*, 0} F^\f{Rng} \cong F^{E_2} H_{*, 0}
        $,
        we get
        \[
            H_{*, 0}(A)
            =
            F^\f{Rng} (1_* \Z \sigma \oplus 2_* \Z \tau)
                / (\sigma \tau - \sigma^3).
        \]
    The statement now follows from \cref{prop:pi_0_r}.
\end{proof}

\begin{lemma}\label{lem:presentation_e2_conn}
    $H_{n, d}^{E_2}(R_\Z, A) = 0$ for
        $\frac{d}{n} < \frac13$.
\end{lemma}

\begin{proof}
    It is known that
        \[
            H_{n, d}^{E_2}(A)
            =
            0
            =
            H_{n, d}^{E_2}(R_\Z)
        \]
        for $d \leq \floor{\frac{n - 1}{2}} - 1$.
    The first equality is clear from the bidegrees $(n, d)$
        of the generators and relators in the presentation.
    The second equality is \cref{cor:e2_vanishing_r}.
    From the long exact sequence for $(R_\Z, A)$,
        it follows that
        \[
            H_{n, d}^{E_2}(R_\Z, A) = 0
            \text{ when $d \leq \floor{\frac{n - 1}{2}} - 1$.}
        \]
    Since the only $(n, d)$ satisfying
        $d > \floor{\frac{n - 1}{2}} - 1$ and
        $\frac{d}{n} < \frac13$
        are $(n, d) = (1, 0), (2, 0), (4, 1)$,
        it remains to show vanishing of $H_{n, d}^{E_2}(R_\Z, A)$ in these cases.

    For any $n \geq 0$,
        the long exact sequence also gives the exact sequence
        \[
            H_{n, 0}^{E_2}(A)
            \tox{\cong}
            H_{n, 0}^{E_2}(R_\Z)
            \to
            H_{n, 0}^{E_2}(R_\Z, A)
            \to
            0,
        \]
        where the first map is an isomorphism by
            \cref{lem:h0_a} and
            the fact that
                $
                    H_{*, 0} \circ Q^{E_2}
                    =
                    Q^{\f{Rng}} \circ H_{*, 0}.
                $
    Thus,
        $
            H_{n, 0}^{E_2}(R_\Z, A) = 0,
        $
        resolving the first two cases.

    The case $(n, d) = (4, 1)$ remains.
    As was noted above,
        $\tau \cdot d$ generates $H_1(\Gamma_{2, 1})$.
    By \cref{prop:homology_n_gamma_g_1},
        $\sigma^2 \cdot u$,
        $\tau \cdot b$, and
        $a = \sigma^2 \cdot d$
        generate $H_1(N\Gamma_{4, 1})$.
    Therefore, $H_{4, 1}(A) \to H_{4, 1}(R_\Z)$ is surjective.
    We have an exact sequence
        \[
            \begin{tikzcd}
                H_{4, 1}(A)
                \ar[r, two heads]
                    & H_{4, 1}(R_\Z)
                    \ardelcenter
                    \ar[r]
                        & H_{4, 1}(R_\Z, A)
                        \ardel[lld]
                \\
                H_{4, 0}(A)
                \ar[r, "\cong"]
                    & H_{4, 0}(R_\Z),
            \end{tikzcd}
        \]
        where the last map is an isomorphism by \cref{lem:h0_a}.
    Thus,
        $
            H_{4, 1}(R_\Z, A) = 0.
        $
    Using \cite[\prop{11.9}]{GKRW19Ek}
        with
            $c = (0, 0, 0, \ldots)$,
            $c_f = (1, 1, 1, \ldots)$,
        the Hurewicz map
        \[
            H_{4, 1}(R_\Z, A)
            \to
            H_{4, 1}^{E_2}(R_\Z, A)
        \]
        is surjective.
    In particular,
        $
            H_{4, 1}^{E_2}(R_\Z, A) = 0
        $
        as desired.
\end{proof}

\subsection{Proof of \cref{thm:secondary_stability} and \cref{thm:secondary_stability_2}}
\label{subsec:secondary_stability_proof}

The construction in \cref{sec:map} gives us an object
    $
        R_\Z \sslash (\tilde\sigma, \tilde\tau)
        \in \sMod_\Z^{\N_0},
    $
    which is weakly equivalent to the linearization of the object
    $
        R \sslash (\tilde\sigma, \tilde\tau)
        \in \SSet_*^{\N_0}
    $
    studied in \cref{sec:map}.

\begin{theorem}\label{thm:quotient_vanishing}
    $
        \tilde H_{n, d}(R_\Z \sslash (\tilde\sigma, \tilde\tau)) = 0
    $
    for $\frac{d}{n} < \frac13$.
\end{theorem}

\begin{proof}
    As a convention,
        $
            \F_0 \defeq \Q.
        $
    For a prime $\p$ or $\p = 0$,
        let
        $
            R_\p
            \defeq
            R_\Z \otimes \F_\p
            \in \Alg_{E_2}(\sMod_{\F_\p}^{\N_0})
        $
        be the change of coefficients to $\F_\p$.
    In turn, it suffices to show
        $
            H_{n, d}(R_\p \sslash (\tilde\sigma, \tilde\tau)) = 0
        $
        for each $\p$ and $\frac{d}{n} < \frac13$.

    Fix $\p$.
    By \cref{lem:presentation_e2_conn} and \cite[\thm{11.21}]{GKRW19Ek},
        there is an $E_2$-CW-approximation
        \[
            A \otimes \F_\p \to Z \tox{\simeq} R_\p,
        \]
        with $E_2$-cells $D_{\F_\p}^{(n_i, d_i)} x_i$ with $\frac{d_i}{n_i} \geq \frac13$
            attached to $A \otimes \F_\p$.
    Let $\f{sk}(Z)$ denote $Z$ filtered with the nonnegative skeletal filtration.
    There is a trigraded spectral sequence
        \begin{equation}\label{eqn:ss_mod_sigma_tau}
            E_{n, p, q}^1
            =
            H_{n, p + q, p}(\L \f{gr}(\f{sk}(Z) \sslash (\sigma, \tau)))
            \implies
            H_{n, p + q}(Z \sslash (\sigma, \tau)),
        \end{equation}
        obtained from the filtered object $\f{sk}(Z) \sslash (\sigma, \tau)$.
    Here,
        \[
            \L \gr :
                \Ho((\sMod_\Z^{\N_0})^{\Z_{\leq}})
                \to
                \Ho((\sMod_\Z^{\N_0})^{\Z_=})
        \]
        is the derived associated graded.
    Since homotopy cofibers commute with homotopy cofibers,
        and since $\f{sk}(Z)$ is filtered by cofibrations,
        we have
        \[
            \L \f{gr}(\f{sk}(Z) \sslash (\sigma, \tau))
            \simeq
            \f{gr}(\f{sk}(Z)) \sslash (\sigma, \tau).
        \]
    By \cite[\lem{12.7(iii)}]{GKRW19Ek} and a variant of \cite[\thm{6.14}]{GKRW19Ek},
        we have
        \begin{equation}\label{eqn:ass_gr}
            \f{gr}(\f{sk}(Z))
            \cong
            0_* (A \otimes \F_\p)
            \oplus E_2(\bigoplus_i (d_i)_* S_{\F_\p}^{n_i, d_i} x_i)
        \end{equation}
        as
        $
            E_2(
                S_{\F_\p}^{1, 0} \sigma
                \oplus S_{\F_\p}^{2, 0} \tau
            )
        $-modules in $\SSet^{\N_0 \times \Z_=}$.
    Filtering
        $
            \f{gr}(\f{sk}(Z))
        $
        by the cell attachment filtration
        (see \cite[(6.5)]{GKRW19Ek}),
        we get an additional spectral sequence
        \begin{equation}\label{eqn:ss_mod_sigma_tau_2}
            \tilde E_{n, p, q}^1
            =
            H_{n, p + q}(\L \f{gr}_\f{cell}(\f{gr}_\f{sk}(Z) \sslash (\sigma, \tau))(p))
            \implies
            H_{n, p + q}(\f{gr}_\f{sk}(Z) \sslash (\sigma, \tau))
        \end{equation}
        computing $E^1_{*, *, *}$,
        where we decorate each $\f{gr}$ with a subscript specifying the filtration.
    As before,
        \begin{align*}
            \L \f{gr}_\f{cell}(\f{gr}_\f{sk}(Z) \sslash (\sigma, \tau))
            &\simeq
            \f{gr}_\f{cell}(\f{gr}_\f{sk}(Z)) \sslash (\sigma, \tau)
            \\
            &\simeq
            E_2(X) \sslash (\sigma, \tau)
        \end{align*}
        as
        $
            E_2(
                S_{\F_\p}^{1, 0} \sigma
                \oplus S_{\F_\p}^{2, 0} \tau
            )
        $-modules,
        where
        \[
            X
            \defeq
            S_{\F_\p}^{1, 0} \sigma
            \oplus S_{\F_\p}^{2, 0} \tau
            \oplus S_{\F_\p}^{2, 1} d
            \oplus S_{\F_\p}^{2, 1} b
            \oplus S_{\F_\p}^{2, 1} u
            \oplus S_{\F_\p}^{3, 1} \lambda
            \oplus S_{\F_\p}^{4, 2} \rho
            \bigoplus_i S_{\F_\p}^{n_i, d_i} x_i,
        \]
        forgetting the internal grading.
    By Cohen's theorem \cite[\thm{16.4}]{GKRW19Ek},
        \begin{equation}\label{eqn:cohen}
            H_{*, *}(E_2(X))
            \cong
            H_{*, *}(E_2^+(X))
            \cong
            W_1(H_{*, *}(X)),
        \end{equation}
        where $W_1(\blank)$ takes free $W_1$-algebra
            (see \cite[16.1]{GKRW19Ek}).
    To be more explicit,
        this is the bigraded commutative algebra
        \[
            W_1(H_{*, *}(X))
            =
            \Lambda_{\F_\p}(L)
        \]
        where $L$ is the trigraded vector space
            with homogeneous basis the Dyer--Lashof monomials $Q_{\F_\p}^I(y)$
                for $I$ satisfying familiar admissibility and excess conditions
                and for $y$ a basic Lie word in
                    $
                        \{
                            \sigma,
                            \tau,
                            d,
                            b,
                            u,
                            \lambda,
                            \rho,
                            x_i | i
                        \}.
                    $
    The homotopy cofibers defining $E_2(X) \sslash (\sigma, \tau)$ have
        associated long exact sequences in homology.
    By \cref{eqn:cohen},
        these degenerate into short exact sequences,
        which then fit into diagrams
        \[
            \begin{tikzcd}[column sep = tiny]
                    & 0 \ar[d]
                        & 0 \ar[d]
                            & 0 \ar[d]
                                &
                \\
                0 \ar[r]
                    & H_{n, *}(E_2(X))
                    \ar[r, "\cdot \sigma"]
                    \ar[d, "\cdot \tau"]
                        & H_{n + 1, *}(E_2(X))
                        \ar[r]
                        \ar[d, "\cdot \tau"]
                            & H_{n + 1, *}(E_2(X) \sslash \sigma)
                            \ar[r]
                            \ar[d]
                                & 0
                \\
                0 \ar[r]
                    & H_{n + 2, *}(E_2(X))
                    \ar[r, "\cdot \sigma"]
                    \ar[d]
                        & H_{n + 3, *}(E_2(X))
                        \ar[r]
                        \ar[d]
                            & H_{n + 3, *}(E_2(X) \sslash \sigma)
                            \ar[r]
                            \ar[d]
                                & 0
                \\
                0 \ar[r]
                    & H_{n + 2, *}(E_2(X) \sslash \tau)
                    \ar[r]
                    \ar[d]
                        & H_{n + 3, *}(E_2(X) \sslash \tau)
                        \ar[r]
                        \ar[d]
                            & H_{n + 3, *}(E_2(X) \sslash (\sigma, \tau))
                            \ar[r]
                            \ar[d]
                                & 0
                \\
                    & 0
                        & 0
                            & 0.
                                &
            \end{tikzcd}
        \]
    Therefore,
        \[
            H_{*, *}(E_2(X) \sslash (\sigma, \tau))
            \cong
            W_1(H_{*, *}(X)) / (\sigma, \tau).
        \]
    We have now computed the $\tilde E^1$-term of \cref{eqn:ss_mod_sigma_tau_2}
        after forgetting the internal grading.

    The slope of a homology class $x$ in bidegree $(n, d)$ is the number
        $
            \frac{d}{n} \in \Q \cup \{\infty\}.
        $
    Products $x \cdot y$ have slope greater than or equal to
        the least of the slopes of $x, y$.
    Browder brackets $[x, y]$ have slope strictly greater than
        the least of the slopes of $x, y$.
    The Dyer--Lashof operations never decrease slope either.
    $\sigma$ and $\tau$ are the only generators in $X$
        that have slope $< \frac13$.
    Thus, only basis elements of $L$ involving $\sigma$ and $\tau$
        can have slope $< \frac13$.
    In fact, the only basis elements of $L$
        that have slope $< \frac13$ are
            $
                \sigma \in W_1(H_{*, *}(X))_{1, 0},
            $
            $
                \tau \in W_1(H_{*, *}(X))_{2, 0},
            $
            $
                [\tau, \tau] \in W_1(H_{*, *}(X))_{4, 1},
            $
            and when $\p = 2$,
            $
                Q^1_{\F_2}(\tau) \in W_1(H_{*, *}(X))_{4, 1}.
            $

    \begin{claim}
        $\tilde E^2_{n, p, q} = 0$
            for $\frac{p + q}{n} < \frac13$.
    \end{claim}

    \begin{claimproof}
        To show this statement,
            we may forget about the internal grading $p$
            by defining
            $
                \tilde E^r_{n, k}
                \defeq
                \bigoplus_{p + q = k}
                    \tilde E^r_{n, p, q}.
            $
        The claim then becomes
            that $\tilde E^2_{n, k} = 0$ when $\frac{k}{n} < \frac13$.
        $\tilde E^2_{*, *}$ is the homology of the free differential graded algebra
            \[
                (\tilde E^1_{*, *}, d^1)
                =
                (\Lambda_{\F_\p}(L / \<\tau, \sigma\>), d^1).
            \]
        To estimate these homology groups,
            we make use of an auxiliary filtration.
        Abuse notation and write
            $
                Q^1_\Z(x) \defeq Q^1_\Z(x) \otimes \F_\p
                \in H_{*, 0}(R_\p).
            $
        Filter $\tilde E^1_{*, *}$
            by giving
                $Q^1_\Z(\tau)$ and
                $\rho$
            filtration $0$,
            and filtering the remaining basis elements by homological degree,
            and extending this filtration multiplicatively.
        This filtration respects the differentials.
        The desirable effect of this filtration is
            that taking the associated graded filters away most of the $d^1$-differential,
            leaving only the differentials that we need.
        In particular,
            the associated graded of this filtration is
            \[
                U
                \defeq
                (\Lambda_{\F_\p}(
                    L / \<\sigma, \tau, Q_\Z^1(\tau), \rho\>
                ), d^1 = 0)
                \otimes
                (\Lambda_{\F_\p}[Q_\Z^1(\tau), \rho], \delta)
            \]
            where
                $\delta(\rho) = Q_\Z^1(\sigma)$ and
                $\delta(Q_\Z^1(\sigma)) = 0$
                (all generators of degree $0$ were killed).
        The first factor of the tensor product has itself as homology.
        If $\p$ is odd,
            the homology of the second factor of the tensor product is $0$.
        If $\p = 2$,
            the homology of the second factor of the tensor product is $\Lambda_{\F_2}[\rho^2]$.
        Now, there is a trigraded spectral sequence
            \[
                \overline E^1_{n, p, q}
                =
                H_{n, p + q, p}(U)
                \implies
                \tilde E^2_{n, p + q}.
            \]
        By the considerations preceding the statement of the claim
            and by properties
                \cref{item:integral_refinement:1} and
                \cref{item:integral_refinement:2}
            of $Q^1_\Z$ stated in the beginning of the section,
            all basis elements for $\overline E^1_{n, p, q}$ have slope
            $
                \frac{p + q}{n} \geq \frac13.
            $
        It follows that $\overline E^1_{n, p, q} = 0$
            when $\frac{p + q}{n} < \frac13$,
            and hence that $\tilde E^2_{n, k} = 0$ when $\frac{k}{n} < \frac13$,
            as desired.
    \end{claimproof}

    The statement now follows from
        the spectral sequences \cref{eqn:ss_mod_sigma_tau}, \cref{eqn:ss_mod_sigma_tau_2} and
        the claim.
\end{proof}

\begin{proof}[Proof of \cref{thm:secondary_stability}]
    We have a cofiber sequence
        \[
            R_\Z \sslash \tilde\tau \otimes S_\Z^{1, 0}
            \tox{\blank \cdot \sigma}
            R_\Z \sslash \tilde\tau
            \to
            R_\Z \sslash (\tilde\sigma, \tilde\tau)
            \tox{\del}
            R_\Z \sslash \tilde\tau \otimes S_Z^{1, 1}
        \]
        in $\sMod_\Z^{\N_0}$.
    Since $H_{n, d}((\RO)_\Z \sslash \tilde\tau) = 0$
        for $\frac{d}{n} < \frac13$ by \cite[\thm{B(i)}]{GKRW19MCG},
        we have
            $
                H_{n, d}(R_\Z \sslash \tilde\tau)
                =
                H_d(N\Gamma_{n, 1}, N\Gamma_{n - 2, 1})
            $
            in the same range.
    From this and \cref{lem:secondary_stabilization_map_id},
        it follows that
            in the range $\frac{d}{n} < \frac13$ and $d \geq 1$,
            the first map in the cofiber sequence induces
                the secondary stabilization map of \cref{def:map1} on $H_{n, d}$,
        and so the statement follows from \cref{thm:quotient_vanishing}.
\end{proof}

\begin{proof}[Proof of \cref{thm:secondary_stability_2}]
    We have a cofiber sequence
        \[
            R_\Z \sslash \tilde\sigma \otimes S_\Z^{2, 0}
            \tox{\blank \cdot \tilde\tau}
            R_\Z \sslash \tilde\sigma
            \to
            R_\Z \sslash (\tilde\sigma, \tilde\tau)
            \tox{\del}
            R_\Z \sslash \tilde\sigma \otimes S_Z^{2, 1}
        \]
        in $\sMod_\Z^{\N_0}$.
    The theorem now follows from
        \cref{lem:secondary_stabilization_map_id_2} and
        \cref{thm:quotient_vanishing}.
\end{proof}

\begin{remark}
    The proof of \cref{thm:secondary_stability_2} in fact recovers
        the best known version of Harer stability \cite[Theorem B(i)]{GKRW19MCG},
        due to the orientable terms in \cref{lem:secondary_stabilization_map_id_2}.
\end{remark}

\begin{remark}\label{rem:g_odd}
    \cref{thm:quotient_vanishing} implies
        a version of \cref{thm:secondary_stability_2} for $g$ odd.
    The resulting statement regards a map of the type
        \[
            H_d(N\Gamma_{g - 2, 1}, N\Gamma_{g - 3, 1} * \Gamma_{(g - 3)/2, 1})
            \to
            H_d(N\Gamma_{g, 1}, N\Gamma_{g - 1, 1} * \Gamma_{(g - 1)/2, 1}).
        \]
    To define the participating relative homology groups,
        one is forced to pick diffeomorphisms
        \[
            S_{h, 1} \oplus N_{1, 1}
            \tox{\cong}
            N_{2h + 1, 1}
        \]
        of which there are no canonical choices (see also \cref{rem:big_cat}).
    Due to this ambiguity,
        we choose not to state a version of \cref{thm:secondary_stability_2} for $g$ odd,
        as it seems less useful.
\end{remark}

\appendix
%! TEX root = main.tex

\section{Classifying spaces of braided monoidal categories}
\label{sec:classifying}

Let $(\mathscr C, \otimes, U, \beta)$ be a braided strict monoidal category.
Its classifying space $B \mathscr C$ inherits the structure of a unital $E_2$-algebra.
One can exhibit this structure in an abstract way by using
    the formal language of model categories and derived Kan extensions;
    see \cite[17.1]{GKRW19Ek}.
However, in \cref{sec:map} we needed to concretely understand the sense in which the
    half Dehn twist of $2$-cubes corresponds to the braiding $\beta$.
For this purpose,
    we describe another, more explicit way to construct
        a unital $E_2$-algebra from a strict braided monoidal category.

%Let $F_n(\R^2)$ denote
%    the space of ordered configurations of points in the plane $\R^2$.
%Note that the symmetric group $\Sigma_n$ acts on $F_n(\R^2)$ by reordering the points.
%Let
%    \[
%        x_n
%        \defeq
%        ((\frac{1}{2n}, \frac12), (\frac{3}{2n}, \frac12), \ldots, (\frac{2n - 1}{2n}, \frac12))
%        \in F_n(\R^2).
%    \]
%Denote
%    by $B_n$ (called the $n$'th \emph{braid group})
%        the group $\pi_1(F_n(\R^2), x_n)$ and
%    by $\tilde B_n$ (called the $n$'th \emph{pure braid group})
%        the group $\pi_1(F_n(\R^2)/\Sigma_n, \overline x_n)$.

\begin{definition}
    $\Sym_0$ is the terminal category.
    For $n \geq 1$,
        $\Sym_n$ is the discrete category on the object set
        \[
            \Sigma_n
            \defeq
            \{
                \text{permutations $\{1, \ldots, n\} \tox{\cong} \{1, \ldots, n\}$}
            \}.
        \]
\end{definition}

\begin{definition}
    The (symmetric) operad $S$ in simplicial sets has
        \[
            S(n) \defeq N\Sym_n
        \]
        where $\Sigma_n$ acts through postcomposition.
    For $n \geq 1$,
        its operadic composition
        $
            S(n) \times S(r_1) \times \cdots \times S(r_n)
            \to
            S(r_1 + \cdots + r_n)
        $
        is induced by the functor
        \begin{align*}
            \Sym_n \times \Sym_{r_1} \times \cdots \times \Sym_{r_n}
            &\to
            \Sym_{r_1 + \cdots + r_n}.
            \\
            (\sigma_0, \sigma_1, \ldots, \sigma_n)
            &\mapsto
            \sigma_{\sigma_0(1)} \oplus \cdots \oplus \sigma_{\sigma_0(n)},
        \end{align*}
        where
            $\oplus$ denotes juxtaposition of permutations.
\end{definition}

\begin{definition}
    $\Braid_0$ is the terminal category.
    For $n \geq 1$,
        the groupoid $\Braid_n$ has
        $
            \Ob(\Braid_n)
            \defeq
            \Sigma_n
        $
        as set of objects
        and
        \[
            \Braid_n(\sigma_1, \sigma_2)
            \defeq
            \{
                b \in B_n
                \mid
                t(b) = \sigma_2 \circ \sigma_1^{-1}
            \}
        \]
        as morphisms,
        where
            $B_n$ is the $n$'th braid group and
            $
                t : B_n \to \Sigma_n = \Ob(\Braid_n)
            $
            is the canonical homomorphism.
    Composition is induced by the group operation of $B_n$,
        \begin{align*}
            \circ :
                \Braid_n(\sigma_2, \sigma_3) \times \Braid_n(\sigma_1, \sigma_2)
                &\to
                \Braid_n(\sigma_1, \sigma_3),
            \\
            (b_1, b_2)
            &\mapsto
            b_1 b_2 \in B_n.
        \end{align*}
    In particular, for $n \geq 1$,
        $
            \Aut_{\Braid_n}(\sigma) = \tilde B_n
        $
        is the pure braid group.
\end{definition}

For $n \geq 1$,
    there is a natural action
    $
        \Sigma_n \acts \Braid_n
    $
    which
        on objects is given by postcomposing by the acting permutation and
        on morphisms is given by relabeling braids.

\begin{definition}
    The (symmetric) operad $B$ in simplicial sets has
        \[
            B(n) \defeq N\Braid_n
        \]
        where $\Sigma_n$ acts through its natural action on $\Braid_n$.
    For $n \geq 1$,
        its operadic composition
        $
            B(n) \times B(r_1) \times \cdots \times B(r_n)
            \to
            B(r_1 + \cdots + r_n)
        $
        is induced by the functor
        \begin{align*}
            \Braid_n \times \Braid_{r_1} \times \cdots \times \Braid_{r_n}
            &\to
            \Braid_{r_1 + \cdots + r_n}.
            \\
            (\sigma_0, \sigma_1, \ldots, \sigma_n)
            &\mapsto
            \sigma_{\sigma_0(1)} \oplus \cdots \oplus \sigma_{\sigma_0(n)},
            \\
            (\id_{\sigma_0}, b_1, \ldots, b_n)
            &\mapsto
            b_{\sigma_0(1)} \oplus \cdots \oplus b_{\sigma_0(n)},
            \\
            (b, \id, \ldots, \id)
            &\mapsto
            b_*
        \end{align*}
        where
            $\oplus$ denotes juxtaposition of braids or permutations, and
            $b_* \in B_{r_1 + \ldots + r_n}$ is block braid induced by $b \in B_n$.
\end{definition}

Let $\mathcal C_k^+$ denote
    the unital little $k$-cubes operad (\cite[Definition 12.1]{GKRW19Ek}).
There is a natural map of operads
    \[
        f : \mathcal C_1^+ \to \mathcal C_2^+,
    \]
    sending $([a_1, b_1], \ldots, [a_n, b_n])$
    to $([a_1, b_1] \times I, \ldots, [a_n, b_n] \times I)$.
Furthermore, there is a weak equivalences of operads
    \[
        g : \mathcal C_1^+ \to S,
    \]
    which sends $([a_1, b_1], \ldots, [a_n, b_n])$ to the permutation $\sigma \in \Sigma_n$
    such that $a_{\sigma(1)} < \ldots < a_{\sigma(n)}$.

\begin{proposition}\label{prop:e2_models}
    $B$ is $\Sigma$-cofibrant and there is a zig-zag of weak equivalences of pairs of operads
        \[
            (B, S)
            \tox{\simeq}
            \cdots
            \fromx{\simeq}
            (\mathcal C_2^+, \mathcal C_1^+).
        \]
\end{proposition}

\begin{proof}
    It is $\Sigma$-cofibrant
        because each action $\Sigma_n \acts B(n)_p$ is free.
    We now exhibit the zig-zag.
    For each $n \geq 0$,
        there is a diagram,
        \begin{equation}\label{dia:zigzag_b}
            \begin{tikzcd}[column sep = normal]
                S(n)
                \ar[d]
                    & \mathcal C_1^+(n)
                    \ar[d, "f_* \circ \eta"]
                    \ar[r, equals]
                    \ar[l, "g", "\simeq"']
                        & \mathcal C_1^+(n)
                        \ar[d, "f_* \circ \eta"]
                        \ar[r, equals]
                            & \mathcal C_1^+(n)
                            \ar[d, "f"]
                \\
                B(n)
                    & N\Pi(\mathcal C_2^+(n), f(\mathcal C_1^+(n)))
                    \ar[l, "\simeq"', "(1)"]
                    \ar[r, "\simeq", "(2)"']
                        & N\Pi(\mathcal C_2^+(n))
                            & \mathcal C_2^+(n).
                            \ar[l, "\simeq"', "(3)"]
            \end{tikzcd}
        \end{equation}
    $\Pi(X)$ denotes the fundamental groupoid of $X$
        and $\Pi(X, A)$ for $A \subseteq X$ denotes
            the full subgroupoid of $\Pi(X)$ on the objects $A$.
    $\eta$ denotes the $1$-truncation map
        $
            \mathcal C_1^+(n)
            \to
            N \Pi(\mathcal C_1^+(n)).
        $
    (1) is induced by the equivalence of categories
        $
            \Pi(\mathcal C_2^+(n), f(\mathcal C_1^+(n)))
            \to
            \Braid_n
        $
        that identifies the various vertical configurations corresponding to the same permutation.
    (2) is induced by
        the essentially surjective inclusion of
            the full subcategory
            $
                \Pi(\mathcal C_2^+(n), f(\mathcal C_1^+(n)))
            $
            into $\Pi(\mathcal C_2^+(n))$.
    (3) is the $1$-truncation map and is a weak equivalence
        since its source is an Eilenberg--MacLane space.
    Ranging $n$, the diagrams assemble into a diagram of operads in $\SSet$,
        exhibiting the desired zig-zag.
\end{proof}

\begin{definition}
    Let $(\mathscr C, \otimes, U, \beta)$ be a braided strict monoidal category.
    For each $n \geq 1$,
        there is a functor
        \begin{align*}
            \Braid_n \times \mathscr C^n
            &\to
            \mathscr C,
            \\
            (\sigma, x_1, \ldots, x_n)
            &\mapsto
            x_{\sigma(1)} \otimes \cdots \otimes x_{\sigma(n)},
            \\
            (\id_\sigma, f_1, \ldots, f_n)
            &\mapsto
            f_{\sigma(1)}
            \otimes
            \cdots
            \otimes
            f_{\sigma(n)},
            \\
            (b : \sigma_1 \to \sigma_2, \id_{x_1}, \ldots, \id_{x_n})
            &\mapsto
            b_*.
        \end{align*}
    Here,
        $
            b_* :
                x_{\sigma_1(1)} \otimes \cdots \otimes x_{\sigma_1(n)}
                \to
                x_{\sigma_2(1)} \otimes \cdots \otimes x_{\sigma_2(n)}
        $
        denotes the braiding isomorphism corresponding to $b \in B_n$.
    There is also a functor
        \[
            \Braid_0 \to \mathscr C
        \]
        sending the unique object in $\Braid_0$ to the unit $U$.
    Taking the nerve $N$,
        we get for each $n \geq 0$ a map of simplicial sets
        \[
            N\Braid_n \times N\mathscr C^n
            \to
            N\mathscr C.
        \]
    These maps endow $N\mathscr C \in \SSet$ with the structure of a $B$-algebra.
\end{definition}

\begin{remark}\label{rem:graded_classifying}
    The construction also works for graded categories.
    Let $(\N_0, +, 0)$ denote the discrete, monoidal category
        whose object set is $\N_0$ with the additive monoidal structure.
    A graded braided strict monoidal category $(\mathscr C, \otimes, U, \beta, r)$ is
        a braided strict monoidal category $(\mathscr C, \otimes, U, \beta)$
        together with a monoidal functor $r : \mathscr C \to \N_0$.
    The graded nerve $N_\gr \mathscr C \in \SSet^{\N_0}$ has
        \[
            N_\gr \mathscr C(n) \defeq N(d^{-1}(n)).
        \]
    The construction above gives a map of graded simplicial sets
        \[
            0_*(N\Braid_n) \otimes (N_\gr\mathscr C)^{\otimes n}
            \to
            N_\gr\mathscr C,
        \]
        where
            $0_* : \SSet \to \SSet^{\N_0}$ is the left adjoint to the projection to $0$, and
            $\otimes : \SSet^{\N_0} \times \SSet^{\N_0} \to \SSet^{\N_0}$ is
                the Day convolution monoidal product.
    This gives $N_\gr\mathscr C$ the structure of an $B$-algebra in $\SSet^{\N_0}$.
\end{remark}

\subsection{Comparison of $E_1$-algebras}

We now prove a little technical lemma
    that we need to ascertain that we can make use of
        the $E_1$-splitting complex theory of \cite{GKRW19Ek}
    notwithstanding how we define our operads and algebras differently.
Let $(\mathscr G, \otimes, U, \beta, r)$ be a
    graded braided strict monoidal category
    in the sense of \cref{rem:graded_classifying}.
Furthermore, assume that $\mathscr G$ is a groupoid.

\begin{definition}\label{def:alg_model}
    Let
        $
            \underline * \in \SSet^{\mathscr G}
        $
        be the $\mathscr G$-graded simplicial set with
        $
            \underline *(x)
            =
            *
        $
        for all $x \in \mathscr G$.
    $\underline *$ admits
        a unique action from $\mathcal C_1^+$.
    Define
        \[
            R
            \defeq
            \L r_*(\underline *)
            =
            r_*(c \underline *)
            \in \Alg_{\mathcal C_1^+}(\SSet^{\N_0}),
        \]
        the derived left Kan extension of $\underline *$ along $r$.
    Here,
        \[
            c : \Alg_{\mathcal C_1^+}(\SSet^{\N_0}) \to \Alg_{\mathcal C_1^+}(\SSet^{\N_0})
        \]
        denotes a cofibrant replacement functor for the projective model structure.
\end{definition}

Recall the map
    $
        g : \mathcal C_1^+ \to S
    $
    defined earlier.

\begin{lemma}\label{lem:alg_model}
    There is a zig-zag of weak equivalences between $R$ and
        $
            g^* N_\gr \mathscr G
        $
        in $\Alg_{\mathcal C_1^+}(\SSet^{\N_0})$
        with the projective model structure.
\end{lemma}

\begin{proof}
    $g^* N_\gr \mathscr G$ may also be described as
        the left Kan extension along $r$ of
        the $\mathcal C_1^+$-algebra $T$
        arising from the obvious monoid with
        $
            T(x) \defeq N(\mathscr G_{/x})
        $
        the (contractible) nerve of the overcategory.
    Since $\mathscr G_{/x}$ is a groupoid for each $x$,
        the map $T \to \underline *$ is a trivial fibration
        in the projective model structure,
        hence the map $c \underline * \to \underline *$ lifts to a weak equivalence
        $
            c \underline * \to T,
        $
        which after left Kan extending along $r$ descends to a weak equivalence
        $
            R \to g^* N_\gr(\mathscr G)
        $
        in $\Alg_{\mathcal C_1^+}(\SSet^{\N_0})$
        since the underlying objects are cofibrant.
\end{proof}

\clearpage
\bibliographystyle{alpha}
\bibliography{ref}

\begin{thebibliography}{GKRW19b}

\bibitem[Bol12]{Boldsen12}
S{\o}ren~K. Boldsen.
\newblock Improved homological stability for the mapping class group with
  integral or twisted coefficients.
\newblock {\em Mathematische Zeitschrift}, 270(1-2):297--329, 2012.

\bibitem[Eps66]{Epstein1966}
David B.~A. Epstein.
\newblock Curves on 2-manifolds and isotopies.
\newblock {\em Acta Mathematica}, 115(0):83--107, 1966.

\bibitem[FM11]{FM11}
Benson Farb and Dan Margalit.
\newblock {\em A {P}rimer on {M}apping {C}lass {G}roups}.
\newblock Princeton University Press, 2011.

\bibitem[GKRW19a]{GKRW19MCG}
S{\o}ren Galatius, Alexander Kupers, and Oscar Randal-Williams.
\newblock {$E_2$}-cells and mapping class groups.
\newblock {\em Publications math{\'{e}}matiques de
  l{\textquotesingle}{IH}{\'{E}}S}, 130(1):1--61, June 2019.

\bibitem[GKRW19b]{GKRW19Ek}
Søren Galatius, Alexander Kupers, and Oscar Randal-Williams.
\newblock Cellular {$E_k$}-algebras.
\newblock 2019.

\bibitem[Har85]{Harer85}
John~L. Harer.
\newblock Stability of the homology of the mapping class groups of orientable
  surfaces.
\newblock {\em Annals of mathematics}, 121(2):215--249, 1985.

\bibitem[Hat91]{Hatcher91}
Allen Hatcher.
\newblock On triangulations of surfaces.
\newblock {\em Topology and its Applications}, 40(2):189--194, 1991.

\bibitem[Iva89]{Ivanov89}
Nikolai~V. Ivanov.
\newblock Stabilization of the homology of teichmuller modular groups.
\newblock {\em Algebra i Analiz}, 1(3):110--126, 1989.

\bibitem[Kor02]{Korkmaz02}
Mustafa Korkmaz.
\newblock Low-dimensional homology groups of mapping class groups: a survey.
\newblock {\em Turkish Journal of Mathematics}, 26(1):101--114, 2002.

\bibitem[Loo13]{Looijenga13}
Eduard Looijenga.
\newblock Connectivity of complexes of separating curves.
\newblock {\em Groups, Geometry, and Dynamics}, 7(2):443--450, 2013.

\bibitem[PS15]{PS15}
Luis Paris and B{\l}a{\.{z}}ej Szepietowski.
\newblock A presentation for the mapping class group of a nonorientable
  surface.
\newblock {\em Bulletin de la Societe mathematique de France}, 143(3):503--566,
  2015.

\bibitem[RW08]{RandalWilliams08}
Oscar Randal-Williams.
\newblock The homology of the stable nonorientable mapping class group.
\newblock {\em Algebraic \& Geometric Topology}, 8(3):1811--1832, 2008.

\bibitem[RW16]{RandalWilliams16}
Oscar Randal-Williams.
\newblock Resolutions of moduli spaces and homological stability.
\newblock {\em Journal of the European Mathematical Society}, 18(1):1--81,
  2016.

\bibitem[Stu06]{Stukow06}
Micha{\l} Stukow.
\newblock Dehn twists on nonorientable surfaces.
\newblock {\em Fundamenta Mathematicae}, 189(2):117--147, 2006.

\bibitem[Wah07]{Wahl07}
Nathalie Wahl.
\newblock Homological stability for the mapping class groups of non-orientable
  surfaces.
\newblock {\em Inventiones mathematicae}, 171(2):389--424, October 2007.

\bibitem[Wah13]{Wahl13}
Nathalie Wahl.
\newblock Homological stability for mapping class groups of surfaces.
\newblock III(26):547--583, 2013.

\end{thebibliography}
\clearpage

\end{document}